\documentclass[12pt]{amsart}

\usepackage{amsfonts, amssymb,amsmath}
\usepackage{epsf}
\usepackage{graphicx}
\newtheorem{theorem}{Theorem}[section]
\newtheorem{lemma}[theorem]{Lemma}
\newtheorem{corollary}[theorem]{Corollary}

\newtheorem{prop}{Proposition}[section]
\theoremstyle{definition}
\newtheorem{definition}{Definition}[section]

\theoremstyle{remark}
\newtheorem{remark}[theorem]{Remark}
\numberwithin{equation}{section}

\begin{document}

\title[Certain Simple Approximately Subhomogeneous algebras]{{\large The
classification of separable simple C*-algebras which are inductive limits of
continuous-trace C*-algebras with spectrum homeomorphic to the closed
interval [0,1]}}
\author{George A. Elliott and Cristian Ivanescu}
\address{Department of Mathematics, University of Toronto, Toronto, ON, Canada   M5S 2E4}
\email{elliott@math.toronto.edu}
\address{Department of Mathematics, University of Toronto, Toronto, ON, Canada   M5S 2E4}
\email{cristian@math.toronto.edu}

%\subjclass{46L35, 46L06.}
%\keywords{K-theory, classification, C*-algebras, inductive limits, real rank one}
\date{September 2005}
\maketitle

\vspace*{15mm}
\noindent
%George A. Elliott,\\
%Department of Mathematics,\\
%University of Toronto,\\
%Toronto, Ontario,\\
%Canada M5S 2E4.\\
%Phone: (416) 978 4804\\
%Fax:   (416) 978 4107\\
%e-mail: elliott@math.toronto.edu\\ \\

\noindent
%Cristian Ivanescu\\ %(corresponding author),\\
%Department of Mathematics,\\
%University of Northern British Columbia,\\
%Prince George, British Columbia,\\
%Canada V2N 4Z9.\\
%Phone: (250) 960 5990 \\
%Fax:   (250) 960 5544\\
%e-mail: ivanescu@unbc.ca\\

%\newpage
%\vspace*{20mm}

\begin{abstract}
{\large A classification is given of certain separable nuclear C*-algebras
not necessarily of real rank zero, namely, the class of separable 
simple C*-algebras which are inductive limits of continuous-trace C*-algebras 
whose building blocks have spectrum homeomorphic to the closed interval $[0,1]$, 
or to a disjoint union of copies of this space. 
%%%%%%%%%%%%%%
% Moreover, a classification of those simple C*-algebras which are hereditary s%ub-C*-algebras of approximate interval algebras is obtained.
%%%%%%%%%%%
Also, the range of the invariant is calculated. }
\end{abstract}

\vspace*{30mm}
\noindent
{\it 1991 Mathematics Subject Classification.} 46L35, 46L06.\\
{\it Key words and phrases.} K-theory, classification, C*-algebras, inductive limits, real rank one.

\newpage
%%%%%%%%%%%%%%%%%%%%%%%%
%\begin{center}
%\small Dedicated to Professor George A. Elliott on the occasion of his sixtieth% birthday
%\end{center}
%%%%%%%%%%%%%%%%%%%%%%%%%%%%%%%%%%%%%%%%%%%%%%%%%%%%%

\section{\protect\large Introduction}

{\large \noindent \indent It is shown in \cite{lin} that an important class
of separable simple crossed product C*-algebras are approximately
subhomogeneous. Recall that a C*-algebra is said to be subhomogeneous if it
is isomorphic to a sub-C*-algebra of $M_{n}(C_{0}(X))$ for some natural
number $n$ and for some locally compact Hausdorff space $X$. An
approximately subhomogeneous C*-algebra, abbreviated ASH algebra, is an
inductive limit of subhomogeneous algebras. \newline
\indent This article contains a partial result in the direction of
classifying all simple ASH algebras by their Elliott invariant.\newline
\indent The first result on the classification of C*-algebras not of real
rank zero was the classification by G. Elliott of unital simple approximate
interval algebras, abbreviated AI algebras (see \cite{ell2}). This result
was extended to the non-unital case independently by I. Stevens (\cite{ste})
and K. Thomsen (\cite{kto}). Also, an interesting partial extension of this
result to the non-simple case was given by K. Stevens (\cite{kst}). It is
worth mentioning that all these algebras are what are referred to as
approximately homogeneous algebras, abbreviated AH algebras, and that the
most general classification result for simple AH algebras was obtained by
Elliott, Gong and Li in \cite{go2}.\newline
\indent One of the first isomorphism results for ASH algebras was the proof
by H. Su of the classification of C*-algebras of real rank zero which are
inductive limits of matrix algebras over non-Hausdorff graphs; see \cite{su}%
. The classification of ASH algebras was also considered in \cite{su0}, \cite%
{rak} and \cite{mig}. (This list of contributions is intended to be
representative rather than complete for the classification of ASH algebras.)%
\newline
\indent An important work on the classification of ASH algebras not of real
rank zero, and in fact one of the first ones, is due to I. Stevens (\cite%
{ste0}). The main result of the present paper is a substantial extension of
Stevens's work, to the class consisting of all simple C*-algebras which are
inductive limits of continuous-trace C*-algebras with spectrum homeomorphic
to the closed interval $[0,1]$ (or to a finite disjoint union of closed
intervals). In particular, the spectra of the building blocks considered
here are the same as for those considered by Stevens. The building blocks
themselves are more general.\newline
\indent The isomorphism theorem is proved by applying the Elliott
intertwining argument.\newline
\indent Inspired by I. Stevens's work, the proof proceeds by showing an
Existence Theorem and a Uniqueness Theorem for certain special continuous
trace C*-algebras. (As can be seen from the proofs, it is convenient to have
a special kind of continuous trace C*-algebra as the domain algebra in both
these theorems. By special we mean having finite dimensional irreducible
representations and such that the dimension of the representation, as a
function on the interval, is a finite (lower semicontinuous) step function.)%
\newline
\indent The present Existence Theorem, Theorem 5.1, differs in an important
way from that of \cite{ste0}, Theorem 29.4.1. In fact Theorem 29.4.1 of \cite%
{ste0} is false, as is shown in Section 5.1 below.\newline
\indent The proof of the present Existence Theorem is an eigenvalue pattern
perturbation, as shown in Section 5, which is similar to the approach used
in \cite{ste0}. (Indeed, once the statement of Theorem 29.4.1 of \cite{ste0}
is corrected, the argument given in \cite{ste0} does not need to be
essentially changed.)\newline
\indent The proof of the present Uniqueness Theorem is different from the
one in \cite{ste0}. It uses the finite presentation of special continuous
trace C*-algebras that was given in \cite{ici0} and \cite{ici}. Also the
present Uniqueness Theorem has the advantage that both the statement and the
proof are intrinsic, i.e., there is no need to say that the building blocks
are hereditary sub-C*-algebras of interval algebras as in \cite{ste0}.%
\newline
\indent In order to apply the Existence and Uniqueness Theorems, it is
necessary to approximate the general continuous trace C*-algebras appearing
in a given inductive limit decomposition by special continuous trace
C*-algebras, as described in \cite{ici}, Theorem 4.15. This is admissible
since in \cite{ici} (and also more generally in \cite{ici0}), it is shown
that these special C*-algebras are weakly semiprojective, i.e., have stable
relations. (A result of T. Loring, Lemma 15.2.2, \cite{lor}, allows one to
conclude that the original inductive limit decomposition can be replaced by
an inductive limit of special continuous trace C*-algebras.)\newline
\indent An important step of the proof of the isomorphism theorem is the
pulling back of the invariant from the inductive limit to the finite stages.
The invariant has roughly two major components: a stable part and a
non-stable part. The pulling back of the stable part is contained in \cite%
{ell2} or \cite{ste0} and is performed in the present situation with respect
to the unital hereditary sub-C*-algebras. The intertwining which is obtained
at the level of the stable invariant will approximately respect the
non-stable part of the invariant on finitely many elements, as pointed out
in \cite{ste0}. To be able to apply the Existence Theorem it is crucial to
ensure that the non-stable part of the invariant is exactly preserved on
finitely many elements (actually, just a single element). It is possible to
obtain an exact preservation of the non-stable invariant on finitely many
elements because one can change the given finite stage algebras in the
inductive limit decomposition in such a way that a non-zero gap arises at
the level of the affine function spaces; see Section 8 below. It is this
non-zero gap that will ultimately guarantee (after passing to subsequences
in a convenient way) the exact intertwining on finte sets of the non-stable
invariant, as shown in Section 9. It is worth mentioning that in the pulling
back of the stable invariant, we must ensure, at the same time that the maps
at the affine function space level are given by eigenvalue patterns. This is
necessary in order to apply the Existence Theorem and is possible by the
Thomsen-Li theorem.\newline
\indent Now all the hypotheses of the Elliott intertwining argument are
fulfilled and in this way the proof of the isomorphism Theorem 3.1 is
completed.\newline
\indent I. Stevens's description of the range of the invariant is also
extended to include the case of unbounded traces (Theorem 3.2).\newline
\indent To conclude, the class of simple inductive limits of
continuous-trace C*-algebras under consideration is compared with the class
of simple AI algebras. %%%%%%%%%%%%%%%%%%%%%%%%%
%\emph{ \bf Acknowledgments.} I would like to thank George A. Elliott for his ma%ny helpful comments on earlier drafts of this paper and for many helpful discus%sions we had. This research was supported by both the Department of Mathematics%, University of Toronto, and by an Ontario Graduate Scholarship.
%%%%%%%%%%%%%%%%%%%%%%%%%%%%%%%%%%%
}

\section{\protect\large The invariant}

{\large %%%%%%%%%%%%%%%%%%%%%%%%%%%%%%%%%%%%%%
% Let us consider the following class of C*-algebras: simple C*-algebras which %are inductive limit of continuous-trace C*-algebras whose building blocks have% their spectra homeomorphic to $[0,1]$ or to a finite disjoint unions of close%d intervals.\\
%%%%%%%%%%%%%%%%%%%%%%%%%%%%%%%%%%%%%%%
\noindent \indent The invariant is similar to the invariant I. Stevens has
used in \cite{ste0}, usually summed up as the Elliott invariant, namely, $%
(K_0(A), \mathrm{AffT}^{+}A,$ $\mathrm{Aff}^{\prime}A)$, where $K_0(A)$ is a
partially ordered abelian group, AffT$^+A$ is a partially ordered vector
space consisting of linear and continuous functions defined on the cone of
traces $T^+A$, Aff$^{\prime}A$ is a certain special subset of $\mathrm{AffT}%
^{+}A$. The special subset $\mathrm{Aff}^{\prime}A$ is the most important
part of the invariant for our purposes, and in an informal way it might be
said to be the non-stable part of the AffT$^{+}A$. Formally, the special
subset Aff$^{\prime}A$ is the convex set obtained as the closure of $\{ 
\hat{a} \in \mathrm{Aff}T^+A|\;a\geq 0, a\in \mathrm{Ped}(A)\;\mathrm{and}%
\;||a||\leq 1\}$ inside Aff$T^+A$, with respect to the topology naturally
associated to a full projection. Here $\hat{a}$ is the linear and continuous
 function defined by the positive element $a$ from the Pedersen ideal by
$\hat{a}(\tau)=\tau(a)$ where $\tau \in T^+A$.
  As shown in \cite{ste0}, Remark 30.1.1
and Remark 30.1.2, the information given by Aff$^{\prime}A$ is equivalent
with that given by the trace-norm map, which is a lower semicontinuous
function $\mu :T^+A \rightarrow \mathbb{R},\; \mu(\tau)= ||\tau||$ and $%
\infty$ if $\tau$ is unbounded.\newline
\indent It is a crucial fact that the trace-norm map is equivalent to the
dimension function in the case of a building block algebra, cf. Section 4
below. The dimension function of a building block (i.e. the function that
assigns to each point in the spectrum of the building block the dimension of
the irreducible representation) can be viewed as a lower semicontinuous
function on the extreme traces normalized on minimal projections in
primitive quotiens and hence we can compare it with functions from Aff$T^+A$%
. Then the subset Aff$^{\prime}A$ is the closure of the set of all affine
functions smaller than the dimension function. Conversely by taking the
supremum over all elements of Aff$^{\prime}A$ we recover the dimension
function in the case of the building blocks.\newline
}

\section{\protect\large The results}

{\large \noindent \indent Using the invariant described above it is possible
to prove a complete isomorphism theorem, namely, }

\begin{theorem}
{\large Let $\mathcal{A}$ and $\mathcal{B}$ be two non-unital simple C*-algebras which are
inductive limits of continuous-trace C*-algebras with spectrum homeomorphic
to $[0,1]$. Assume that \newline
\indent 1. there is a order preserving isomorphism $\psi _0 : K_0(\mathcal{A}%
) \rightarrow K_0(\mathcal{B}),$\newline
\indent 2. there is an isomorphism $\psi _T : \mathrm{Aff}T^+{} \rightarrow 
\mathrm{Aff}T^+\mathcal{B}$, such that 
\begin{equation*}
\psi_T(\mathrm{Aff}^{\prime}\mathcal{A})\subseteq \mathrm{Aff}^{\prime}{},
\end{equation*}
and \newline
\indent 3. the two isomorphisms are compatible: 
\begin{equation*}
\widehat{ \psi_0([p])} = \psi_T( \widehat{[p]}),\; [p] \in K_0(\mathcal{A}).
\end{equation*}%
\newline
Then there is an isomorphism of the algebras $\mathcal{A}$ and $\mathcal{B}$ that induces the
given isomorphism at the level of the invariant. }
\end{theorem}

{\large %%%%%%%%%%%%%%%%%%%%%%%%
%\indent As a corollary of this and the inductive limit decomposition given in %Theorem 10.2, we have
%\begin{theorem} Let $\mathcal{A}$ and $\mathcal{B}$ be two hereditary sub-C*-a%lgebras of a simple approximately interval algebra. Assume that\\ 
%\indent 1. there is a order preserving isomorphism $ \psi _0 : K_0(\mathcal{A}%) \rightarrow K_0(\mathcal{B}),$\\
% \indent 2. there is an isomorphism $ \psi _T : \mathrm{Aff}T^{+}\mathcal{A} \%rightarrow \mathrm{Aff}T^{+}\mathcal{B}$, such that $$ \psi_T(\mathrm{Aff}' \m%athcal{A})\subseteq \mathrm{Aff}'{\mathcal B},$$ and \\
% \indent 3. the two isomorphisms are compatible: 
% $$ \widehat{ \psi_0([p])} = \psi_T( \widehat{[p]}),\; [p] \in K_0({\mathcal A%}).$$\\
% Then there is an isomorphism of the algebras ${\mathcal A}$ and ${\mathcal B}%$ that induces the given isomorphism at the level of the invariant.
%\end{theorem}
%%%%%%%%%%%%%%%%%%%%%%%%%%%%%
\noindent \indent A description is given of the range of the invariant. More
precisely, the following theorem is proved: }

\begin{theorem}
{\large Suppose that $G$ is a simple countable dimension group and $V$ is
the cone associated to a metrizable Choquet simplex. Let $\lambda:V
\rightarrow Hom^+(G,R)$ be a continuous affine map which takes extreme rays
into extreme rays. Let $f:V \rightarrow [0,+ \infty]$ be an affine lower
semicontinuous map, zero at zero and only at zero. Then $(G, V, \lambda,f)$
is the invariant of some simple non-unital inductive limit of
continuous-trace C*-algebras whose spectrum is the closed interval $[0,1]$. }
\end{theorem}

{\large %%%%%%%%%%%%%%%%%%%%%%%%%%%%%%%%%%%%%%
}

\section{\protect\large Special continuous trace C*-algebras with spectrum
the interval [0,1]}

{\large \noindent \indent  In this section we will introduce some
terminology. A very important piece of data that we shall consider is a map
that assigns, to each class of irreducible representations, the dimension of
a representation from that class. Roughly speaking, the dimension function
can be thought of as the non-stable part of the invariant when restricted to
the building blocks. }

\begin{definition}
{\large Let $A$ be a C*-algebra and let $\hat{A}$ denote the spectrum of $A$%
. Then the $dimension$ $function$ is the map from $\hat{A}$ to $\mathbb{R}
\cup +\infty $, 
\begin{equation*}
\pi \mapsto \mathrm{dim}(H_{\pi}), 
\end{equation*}
where by dim$(H_{\pi})$ we mean the dimension of the irreducible
representation $\pi$. }
\end{definition}

{\large \noindent \indent It was shown in \cite{ici}, Theorem 4.13, that the
dimension function is a complete invariant for continuous trace C*-algebras
with spectrum the closed interval $[0,1]$. Also concrete examples were
constructed for each given dimension function, cf. Section 7 of \cite{ici}.%
\newline
\indent Therefore given a lower semicontinuous integer valued (i.e., a
``dimension function'') which is finite-valued and bounded we can exhibit a
continuous trace C*-algebra 
\begin{equation*}
\left( 
\begin{array}{cccccc}
C_{0}(A_n) & C_{0}(A_n) & C_{0}(A_n) & \dots & C_{0}(A_n) &  \\ 
C_{0}(A_n) & C_{0}(A_{n-1}) & C_{0}(A_{n-1}) & \dots & C_{0}(A_{n-1}) &  \\ 
C_{0}(A_n) & C_{0}(A_{n-1}) & C_{0}(A_{n-2}) & \dots & C_{0}(A_{n-2}) &  \\ 
\vdots & \vdots & \vdots & \ddots &  &  \\ 
C_{0}(A_n) & C_{0}(A_{n-1}) & C_{0}(A_{n-2}) & \dots & C[0,1] & 
\end{array}
\newline
\right) \subseteq M_n \otimes C[0,1]. 
\end{equation*}
whose dimension function is the given function. Here $A_n \subseteq A_{n-1}
\subseteq \dots \subseteq [0,1]$ and each $A_i$ is an open subset of $[0,1]$%
. Moreover any trace on such an algebra is of the form $tr\otimes \nu$,
where $tr$ is the usual trace normalized on minimal matrix projections and $%
\nu$ is a finite measure on $[0,1]$. The extreme traces are parameterized by 
$t \in [0,1]$, and are given as $(tr \otimes \delta_{t})_{t\in [0,1]}$,
where $\delta_{t}$ is the normalized point mass at $t$. Then the trace norm
map is equal to the dimension function when restricted to the extreme
traces. To see that the trace norm map is equivalent to the special subset $%
\mathrm{Aff}^{\prime}()$ of the affine function space $\mathrm{Aff}T^+()$ we
repeat the proof of I. Stevens from \cite{ste0}, Remark 30.1.1 and Remark
30.1.2.\newline
\indent Inspired by a construction of I. Stevens in \cite{ste0} we make }

\begin{definition}
{\large A continuous-trace C*-algebra whose spectrum is $[0,1]$ will be
called a special continuous-trace C*-algebra if its dimension function is a
finite-valued $finite$ $step$ function: there is a partition of $[0,1]$ into
a finite union of intervals such that the dimension function is finite and
constant on each such subinterval. }
\end{definition}

\begin{remark}
{\large Let $A$ be a continuous trace C*-algebra with spectrum $[0,1]$ and
with dimension function $d:[0,1] \rightarrow \mathbb{N} \cup \{ + \infty
\}$. There exists a projection-valued function that if composed with the
rank function gives rise to the dimension function $d$. To see this first we
notice that because the Dixmier-Douady invariant of $A$ is trivial, the
C*-algebra $A$ is a continuous field of elementary C*-algebras over $[0,1]$,
where the fibers are hereditary sub-C*-algebras of the algebra of compact
operators. Then take the unit of the hereditary sub-C*-algebra in each
fiber. In this way we construct a projection-valued function which is lower
semicontinuous. By composing this constructed projection-valued function
with the rank function we get the dimension function $d$. }
\end{remark}

{\large %%%%%%%%%%%%%%%%%%%%%%%%%%%%%% 
%\begin{remark} Let $P$ be the unit of the bidual of a continuous trace C*-algeb%ra $A$ with spectrum the closed interval $[0,1]$. Then by $\hat{P}$ we denote t%he lower semicontinuous real-valued function $\hat{P}: T^+A \rightarrow \mathbb%{R}$,
%$$\hat{P}(\tau)=\tau(P).$$
%\end{remark}
%%%%%%%%%%%%%%%%%
}

\begin{remark}
{\large A priori our definition for a special sub-C*-algebra is more general
than I. Stevens's definition. As it is shown in \cite{ici}, any special
sub-C*-algebra in our sense is isomorphic to a special sub-C*-algebra in I.
Stevens's sense. }
\end{remark}

\begin{remark}
{\large It was shown in \cite{ici} that special continuous trace C*-algebras
are finite presented and weakly semiprojective. Also a stronger result was
proven in \cite{eir}, namely that special continuous trace C*-algebras are
strongly semiprojective. }
\end{remark}

{\large %%%%%%%%%%%%%%%%%%%%%%%%%%%%%%%%%%%%%%%%%%
}

\section{\protect\large Balanced inequalities and the Existence Theorem}

{\large \noindent \indent The proof of the isomorphism theorem 3.1 is based
on the Elliott intertwining argument. Among the main ingredients of this
procedure are the Existence Theorem that will be described below as well as
the Uniqueness Theorem that is presented in Section 6.\newline }
{\large \noindent \indent It is worth noticing that for the Existence
Theorem and the Uniqueness Theorem we require that the inequalities are
balanced, i.e., independent of the choice we make for the normalization of
the affine function space. We normalize the affine function spaces with
respect to a full projection. Even though we fix a projection in the domain
algebra for both the Existence Theorem and the Uniqueness Theorem, this
choice does not make any difference when we apply the theorems to obtain an
approximate commuting diagram. As was pointed out to us by Andrew Toms, we
only need to consider a compatible family of projections when we go through
the whole proof, provided that a corresponding projection is chosen in the
codomain algebra. In fact,  we can state the theorems without mentioning the
choices of the projections as long as their }$K_{0}${\large \ -classes \ are
compatible with respect to the }$K_{0}$ -{\large map under consideration
even though they exist and some choices of them will be used during the
proof.\newline
\indent To be able to focus on the new aspects of the present Existence
Theorem as opposed to the Existence Theorem for unital continuous trace
C*-algebras proved by Elliott in \cite{ell2}, we will both state the theorem
and prove it in terms of so-called eigenvalue pattern maps. In our situation
an eigenvalue pattern map is a positive unital map from $C([0,1])$ to $%
C([0,1])$ which is a finite sum of *-homomorphisms from $C([0,1])$ to $%
C([0,1])$. Using the Gelfand theory each such *-homomorphism is given by a
continuous function from $[0,1]$ to $[0,1]$. As follows from the
intertwining of the invariant and will be explained below, Section 9, one
can always obtain a (non-necessarily compatible) eigenvalue patterns maps.%
\newline
\indent The proof of the Existence Theorem is obtained by perturbing an
eigenvalue pattern map between the affine function spaces in a such a way
that it defines an algebra map between the building blocks.\newline
}

{\large %%%%%%%%%%%%%%%%%%%%%%%%%%%
%\indent Note that by an application of the well-known theorem of Thomsen-Li on%e can easily get an eigenvalue pattern maps between affine function spaces wit%h any large multiplicity.\\
%%%%%%%%%%%%%%%%%%%%%%%%%%%%%%%%
}

\begin{theorem}
{\large Let $A$ be a special building block and by $d_A$ denote the
dimension function of $A$. Let a finite subset $F$ contained in $\mathrm{Aff}%
T^{+}A$, and $\epsilon >0$ be given. There is $f^{\prime}\in \mathrm{Aff}%
^{\prime}A$ such that for any special building block $B$ with dimension
function $d_B$, and maps $k: D(A) \rightarrow D(B)\;\mathrm{and}\;T:\mathrm{%
Aff}T^{+}A \rightarrow \mathrm{Aff}T^{+}B$ verifying the conditions\newline
1. $k$ has multiplicity $M_k$,\newline
2. $T$ is given by an eigenvalue pattern and has the property 
\begin{equation*}
T(f^{\prime}) \leq d_B,
\end{equation*}
3. $k$ and $T$ are exactly compatible, i.e., 
\begin{equation*}
\widehat{k ( [r])}= T (\hat{[r]}),
\end{equation*}
%%%%%%%%%%%%%%%%
% \leq \frac{1}{N}||\hat{r}||,\;[r]\in D(A),$$
%%%%%%%%%%%%%%%%%
there is a homomorphism $\psi : A \rightarrow B$ such that $k =\psi_0$ and 
\begin{equation*}
||(T-\psi_T)a||_{\widehat{k(p)}} \leq \epsilon ||a||_{\hat{p}}, \;\; a \in F.
\end{equation*}
}
\end{theorem}

\begin{remark}
{\large Recall that Aff$T^{+}A$ is a Banach space with a norm given by $%
||f||_p= \mathrm{sup}\{|f(\tau)|\;| \;\tau(p)=1,\;\tau \in T^{+}A\}$, where $f \in 
\mathrm{Aff}T^{+}A$ and $p$ is a fixed full projection of $A$. In addition,
using the norm we just defined, Aff$T^{+}A$ is identified with $C([0,1])$.
This identification allows us to compare in the supremum norm the dimension
function and elements of Aff$T^{+}A$. Also the norm of Aff$T^{+}B$ is
defined with respect to a projection from $B$ which is Murray-von Neumann
equivalent to $k(p)$. Since our inequalities at the level of the affine
function spaces are balanced, which is the only theorem that makes sense, in
particular they are independent of the choice of the projection $p$. }
\end{remark}

\begin{proof}
{\large \indent The idea of the proof is to choose in a clever way a
function $f^{\prime}$ and then change within the given tolerance the
eigenvalue functions that appear in the eigenvalue pattern $T$ so that the
image of the dimension function $d_A$ under the new eigenvalue pattern is
smaller than or equal the dimension function of the algebra $B$, as desired.%
\newline
\indent Let $\epsilon >0$ and a finite set $F \subset \mathrm{AffT}^+A$ be
given.\newline
\indent As already mentioned it is a crucial step how $f^{\prime}$ is
chosen. There is no loss in generality if we assume that the dimension
function $d_A$ has only one discontinuity point, $t_0 \in [0,1]$. }

%\begin{center}
%{\large \includegraphics [height = 3 cm] {dim.eps}\newline
%Figure 1. Dimension function $d_A$ }
%\end{center}

\begin{center}
\begin{picture}(400,200)(-50,0)
\put(145,57){]}
\put(145,107){(}
\put(50,60){\line(1,0){95}}
\put(145,110){\line(1,0){95}}
\put(50,10){\vector(1,0){190}}
\put(50,10){\vector(0,1){185}}
\end{picture}

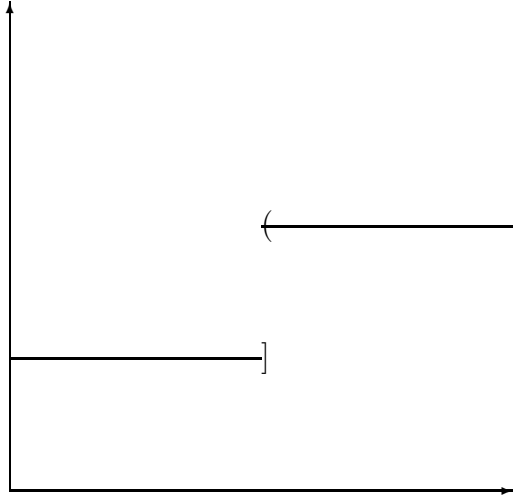
Figure 1.  Dimension function $d_A$.
\end{center}

{\large Choose $f^{\prime}$ to be a continuous function such that $%
f^{\prime}(t)=d_A(t)$ for $t \in [0,t_0-\delta] \cup [t_0+\delta,1],
f^{\prime}(t) \leq d_A(t)$ for $t \in [0,1]$, and $f^{\prime}(t_0)=d_A(t_0)$%
, where $\delta \leq \frac{\epsilon}{2M_{k}^2}$. 
%%%%%%%%%%%%%%%%%%%%%%%%%%%%%%%%
% Then on the one hand we require $f'$ to be equal with $d_A$ where $d_A$ is co%ntinuous. On the other hand $f'$ must differ from $d_A$ on a neighbourhood of %the singularity point $t_0$ which we require to be as small as needed to guara%ntee that the variation of any function $g \in F$ on this neighbourhood is at %most $\epsilon$. In addition we choose $f'$ in a such a manner that $f'$ takes% a strictly smaller value at $t_0$ than the value of $d_A$ at $t_0$.\\
%%%%%%%%%%%%%%%%%%
%\indent An application of the well-known theorem of Thomsen-Li will provide $N%$ and an eigenvalue pattern $T$ with multiplicity $M_k$.\\ 
%%%%%%%%%%%%%%%%%%%%
\indent Hence $f^{\prime}$ is a continuous function defined on the interval $%
[0,1]$ which approximate $d_A$, namely $f^{\prime}$ is equal to $d_A$ except
on a small neighbourhood around the discontinuity point. %%%%%%%%%%%%%%%
% Denote the neighbourhood around $t_0$ by $(t_0-\delta,t_0+\delta)$, where $\d%elta$ is to be specified.
%%%%%%%%%%%%%%%%%%%%%%%%%
}

\begin{center}
\begin{picture}(400,200)(-50,0)
\put(145,57){]}
\put(145,107){(}
\put(50,60){\line(1,0){95}}
\put(145,110){\line(1,0){95}}
\put(50,57){\line(1,0){95}}
\put(159,107){\line(1,0){81}}
\put(145,57){\line(1,3){17}}
\put(50,10){\vector(1,0){190}}
\put(50,10){\vector(0,1){185}}
\end{picture}

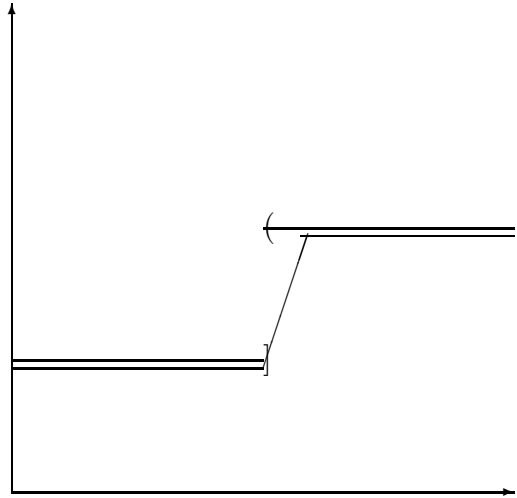
Figure 2.  Graph of $f^{\prime}$.
\end{center}

{\large \indent Next we proceed by showing how to change the eigenfunctions
such that a desired eigenvalue pattern is obtained. We will carry out this
procedure in a very special case, namely all the eigenfunctions are assumed
to be the identity function.\newline
}

{\large %%%%%%%%%%%%%%%%%%%%%%%%%%%%%%
%\indent Although not a necessary step, it is convenient to reduce the problem %to the case where only one eigenvalue function at a time is in $(t_0-\delta,t_%0+\delta)$ where $f'$ differs from $d_A$. To see that this is possible one can% see from the following examples which are relevant for the general case.\\
%\indent Note that there are $M_k$ eigenfunctions including multiplicity. We wi%ll change them slightly such that at each point in the interval $(t_0-\delta,t%_0+\delta)$ there is at most one eigenfunction.
% The changes will become clear if one looks at the following representative ca%se and they correspond to composing the eigenfunction with suitable chosen con%tinuous functions defined from $[0,1]$ to $[0,1]$.
%%%%%%%%%%%%%%%%%%%%%% 
}
\begin{center}
\begin{picture}(400,200)(-50,0)
%\put(145,57){]}
%\put(145,107){(}
\put(50,60){\line(1,0){190}}
\put(50,110){\line(1,0){190}}
\put(50,10){\line(1,1){180}}
\put(50,13){\line(1,1){180}}
\put(50,10){\vector(1,0){190}}
\put(50,10){\vector(0,1){185}}
\end{picture}

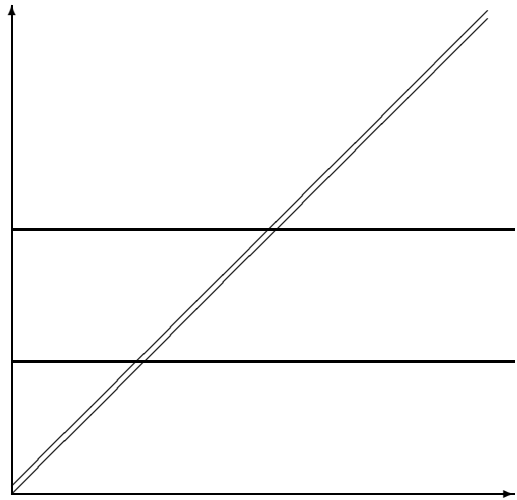
Figure 3.  Eigenfunction $\lambda$.
\end{center}

%\begin{center}
%{\large \includegraphics [height = 3 cm] {e.eps}\newline
%Figure 3. Eigenfunction $\lambda$ }
%\end{center}

{\large \indent In the above picture we have the original eigenvalue
function $\lambda$ which is the identity map. We define a new eigenvalue
function as the picture shows below, Figure 4. More precisely the new
eigenvalue function $\hat{\lambda}:[0,1] \rightarrow [0,1]$, $\hat{\lambda}%
(t)=t$ for $t \in [0,t_0-\delta) \cup (t_0+ \delta+ \delta_{t_0}, 1]$, $\hat{%
\lambda}(t)=t_0-\delta$ for $t\in [t_0-\delta,t_0+\delta]$, the linear map $%
\hat{\lambda}(t)= t_0-\delta + (t-t_0-\delta)\frac{2\delta+\delta_{t_0}}{%
\delta_{t_0}}$ for $t \in [t_0+\delta,t_0+\delta+\delta_{\rho}]$, where $%
\delta_{t_0}$ is a strictly positive number such that $t_0 + \delta +
\delta_{t_0}\leq 1$. }

\begin{center}
\begin{picture}(400,200)(-50,0)
%\put(145,57){]}
%\put(145,107){(}
\put(50,60){\line(1,0){190}}
\put(50,110){\line(1,0){190}}
\put(50,10){\line(1,1){47}}
\put(97,57){\line(1,0){40}}
\put(137,57){\line(1,3){20}}
\put(156,115){\line(1,1){75}}
\put(50,13){\line(1,1){180}}
\put(50,10){\vector(1,0){190}}
\put(50,10){\vector(0,1){185}}
\end{picture}

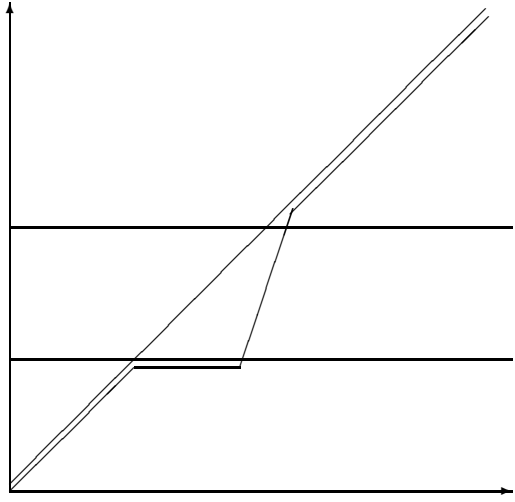
Figure 4.  Eigenfunction $\hat{\lambda}$.
\end{center}

%\begin{center}
%{\large \includegraphics [height = 3 cm] {e2.eps}\newline
%Figure 4. Eigenfunction $\hat{\lambda}$ }
%\end{center}

{\large \indent A short computation or a geometric argument shows that the
difference $||\lambda-\hat{\lambda}||_{\infty}=2\delta$.\newline
\indent Moreover the dimension function $d_A$ evaluated on the perturbed
eigenvalue $\hat{\lambda}$ is smaller then $f^{\prime}$ evaluated on the
given eigenvalue $\lambda$ 
\begin{equation*}
d_A(\hat{\lambda}(t)) \leq f^{\prime}(\lambda(t)).
\end{equation*}
\indent Hence by hypothesis 2 we have 
\begin{equation*}
\sum_{i=1}^{M_k}d_A \circ \hat{\lambda} \leq \sum_{i=1}^{M_k}
f^{\prime}\circ \lambda \leq d_B.
\end{equation*}
%%%%%%%%%%%%%%%%%%%%%%%%%%%%%%%
% Moreover the change is performed such that the new eigenvalue pattern has sma%ller values than the given one when evaluated on the dimension function. Namel%y, we choose to move the eigenvalues towards the subinterval of $(t_0-\delta,t%_0+\delta)$ where $d_A$ takes smaller values. Note that there are at most thre%e subintervals with the convention that the middle one has length zero and cor%responds to the discontinuity point.\\ 
%\indent For more than two eigenfunctions we proceed as in the picture below, e%ach eigenfunction being composed with a similar function as $\rho$ above only %modified in a suitable sense
%\begin{center}
%\includegraphics [height = 3 cm] {en.eps}
%\end{center}
%%%%%%%%%%%%%%%%%%%%%%%%%%%
}

{\large %%%%%%%%%%%%%%%%%%%%%%%%%%
}

{\large \indent Here we say that one dimension function is smaller than
another one if the relation holds pointwise.\newline
\indent The change of the eigenvalues is small because of the choice of $
\delta $ 
\begin{equation*}
||(T_{\lambda }-T)(a)||_{\widehat{k(p)}}=
\sum_{i=1}^{M_{k}}||a \circ (\hat{\lambda_{i}}-\lambda _{i})||_{\widehat{k(p)}}
\end{equation*}
\begin{equation*}
=\sum_{i=1}^{M_{k}} \mathrm{sup}\{|a \circ (\hat{\lambda_{i}}-\lambda _{i})(\tau)| 
 \; | \; \tau(k(p))=1, \tau \in T^+A \}
\end{equation*}
\begin{equation*}
=\sum_{i=1}^{M_{k}} \mathrm{sup}\bigg\{M_k \bigg|a \circ (\hat{\lambda_{i}}-\lambda _{i})
 \bigg(\frac{1}{M_k}\tau\bigg)\bigg| \; | \; \tau(p)=M_k, \tau \in T^+A \bigg\}
\end{equation*}
\begin{equation*}
=\sum_{i=1}^{M_{k}}M_k||a \circ (\hat{\lambda_{i}}-\lambda _{i})||_{\widehat{p}}
\leq 2\delta M_{k}^2||a||_{\hat{p}}\leq \epsilon ||a||_{\hat{p}},\;a\in F.
\end{equation*}
}
{\large
\indent To obtain the inequality above we used the linearity of the function 
$a \circ (\hat{\lambda_{i}}-\lambda _{i})$ and that an extreme trace
$\tau$ in $T^+ A$ has the property that $\tau (k(p))=1$ if and only if 
$\tau(p) = M_k$. }

{\large \indent We claim that the argument for the special case shown above
can be extended to the case of piecewise linear eigenfunctions which is
known to be equivalent to the general case of continuous eigenfunctions that
arise in the inductive limits of interval algebras (see for instance \cite%
{ell2}). }

{\large %%%%%%%%%%%%%%
% Therefore $||(T-T_{\lambda})(a)|| \leq ||a||\epsilon$, $a\in F$.\\
%%%%%%%%%%%%%%%
%%%%%%%%%%%%%%%%%%%%%%%%%%%%%%%
%\indent Now we apply the Existence Theorem for unital full matrix algebras ove%r the interval for the eigenvalue pattern $(\lambda_{i}')$. Restrict the map c%onstructed to the algebra $A$. Then the hereditary sub-C*-algebra generated by% the image has the dimension function equal to $T_{\lambda'}(d_A)$.\\
%\indent  Now we apply the isomorphism theorem (Theorem 6.13) for building bloc%ks and get an inclusion map from the hereditary sub-C*-algebra generated by im%age of $A$ into $B$.\\
%\indent Moreover the map from $A$ to the hereditary sub-C*-algebra generated b%y its image has all the desired properties of the Existence Theorem and by com%posing with an inclusion map into $B$ we don't destroy these properties. There%fore we obtain the desired map between $A$ and $B$. 
%%%%%%%%%%%%%%%%%%%%%%%%%%%%%%%%%%%%%%%%%%%%%%%%
}
\end{proof}

\subsection{\protect\large An exact inequality is necessary between the
non-stable part of the invariant}

{\large \indent As mentioned in the introduction, the Theorem 29.4.1 of \cite%
{ste0} is false. To prove the Existence Theorem it is fundamental to have an
exact inequality between the non-stable part of the invariant at the level
of the affine function space, i.e., $T(f) \leq d_B$ for some continuous
affine function $f \leq d_A$. A weaker inequality is required in the
statement of the Existence Theorem of \cite{ste0}, Theorem 29.4.1, i.e., $%
T(f) \leq d_B(1+\delta)$ for some small $\delta >0$. Therefore it is
possible to construct a counterexample to the I. Stevens Existence Theorem.
This counterexample is already assuming that the positive linear map $T$ is
given by an eigenvalue pattern. To reduce the proof of Theorem 29.4.1 of %
\cite{ste0} to an eigenvalue pattern problem, one needs an extra assumption
in hypothesis 2, for instance a positive gap $\eta >0$ in the other side of
the inequality described above $T(f) + \eta \leq d_B(1+ \delta)$.\newline
\indent Next we describe the counterexample. Let $d_{\mathcal{A}}$ be the
lower semicontinuous function defined on $[0,1]$ which is equal to $2$ on
the subintervals $[0,1/2)$ and $(1/2,1]$, and equal to $1$ at $1/2$. Let $%
\epsilon_0$ be such that $0< \epsilon_0 <1/4$ and $F=\{a_1(t)=t\}$. Let $f$
be a continuous function which approximates $d_{\mathcal{A}}$. Since they
can not be equal everywhere around $1/2$, we can assume that $f(t) < 2 = d_{%
\mathcal{A}}(t)$ for all $t$ in $(1/2-\eta, 1/2+\eta)$, where $\eta >0$ can
be chosen as small as needed.\newline
\indent Let $\delta>0$ be given. There exists a positive integer $M_k$ such
that $\frac{1}{2M_k-1}< \delta$. Then choose $T$ to be defined by $M_k$
eigenvalue functions $(\lambda_i)_{i=1,\dots,M_k}$, all being the identity
functions, $\lambda_i(t)=t$, for all $i=1,\dots,M_k$. Next choose $B$ to be
a continuous trace C*-algebra with dimension function constant equal to $%
2M_k-1$.\newline
\indent Note that the hypothesis 2 of the Existence Theorem 29.4.1 from \cite%
{ste0} holds 
\begin{equation*}
T(f)(t)= \sum_{i=1}^{M_k}f\circ \lambda_i (t) \leq 2M_k \leq (1+\delta)d_{%
\mathcal{B}(t)}.
\end{equation*}
\ Now we claim that among all perturbations of $T$ which are within the
given $\epsilon_0$ with respect to the finite set $F$, the particular one $P$
which is given by the continuous eigenfunctions $(\mu_i)_{i=1,\dots,M_k}$
that have the property $\mu_i(t)=1/2$ for $t \in (1/2-\eta, 1/2+\eta)$, is
the smallest in the sense that the value of $P(d_{\mathcal{A}})$ is the
smallest. Here it is important to notice that because $\epsilon_0 < 1/4$ it
forces that $(\mu_i)_i(t)= \lambda(t)=t$ for $t$ close to $0$ and $1$
including $0$ and $1$. In particular we have $(\mu_i)(0)=\lambda_i(0)=0$.
Therefore 
\begin{equation*}
P(d_{\mathcal{A}})(0)=\sum_{i=1}^{M_k}d_{\mathcal{A}}(%
\mu_i(0)=2M_k>2M_k-1=d_{\mathcal{B}}(0).
\end{equation*}
%%%%%%%%%%%%%%%%%%%%%%%%%%%%%%%%%%%%%%%%%%%%%%%%% 
%Assume there exists $\psi:\mathcal{A} \mapsto \mathcal{B}$ such that $\psi_0=K%$ and $||(\psi_{T}-T)(a)|| < \epsilon ||a||$, for $a \in F$. The equality of t%he maps at the level of $K_0$ implies that $\psi$ is defined by $N$ eigenvalue%s pattern. The inequality at the level of the affine function spaces implies t%hat  the moments of the set of eigencalues that defines $\psi$ and the moments% of the eigenvalues of $T$ are close. Then using Viete relations we get two po%lynomials with close coefficients which in turn implies the roots of the two p%olynomials are close. Therefore the eigenvalue set of $\psi$ can be paired wit%h the eigenvalues set of $T$ which means that the set of eigenvalues of $\psi$% have their image close to the values $1/2-\frac{1}{100000}$. This shows that %$\psi_T(d_{\mathcal{A}}=4N>2N$ or in other words there is no $\psi:\mathcal{A}% \mapsto \mathcal{B}$.\\
%%%%%%%%%%%%%%%%%%%%%%%%%%%%%%55555
\indent Therefore we cannot perturb the eigenfunctions to obtain a compatible
eigenvalue pattern and the Existence Theorem as stated in \cite{ste0} cannot
be proved. }

\section{\protect\large Uniqueness Theorem}

{\large \indent It is important to notice that the conclusion of the Existence
Theorem is part of the hypothesis of the Uniqueness Theorem; this makes
sense since all inequalities are balanced (i.e. independent of the choice of
projection with respect to which the normalization is done).}

\begin{theorem}
{\large Let A be a special continuous-trace C*-algebra, $F\subset A$ a
finite subset and $\epsilon >0$. Let $B$ be a special continuous-trace
C*-algebra and $\psi ,\varphi :A\rightarrow B$ be maps with the following
properties:\newline
\indent1. $\varphi _{0}=\psi _{0}:K_{0}(A)\rightarrow K_{0}(B)$,\newline
\indent2. $\psi $ and $\varphi $ have at least the fraction $\delta $ of
their eigenvalues in each of the $d$ consecutive subintervals of length $%
\frac{1}{d}$ of $[0,1]$, for some $d>0$ such that for $\hat{r_{i}}$ the
functions equal to $0$ from $0$ to $\frac{i}{d}$, equal to $1$ on $[\frac{i+1%
}{d},1]$ and linear in between, for each $0\leq i\leq d$, $||(\varphi
_{T}-\psi _{T})(\hat{r_{i}})||_{K(p)}<\delta ||\hat{r_{i}}||_{p}$, with
respect to the norm of $\mathrm{Aff}T^{+}B$,\newline
\indent Then there is an approximately inner automorphism of $B$, $f$, such
that 
\begin{equation*}
||(\psi -f\varphi )(a)||<\epsilon ,\;\;\;a\in F
\end{equation*}%
}
\end{theorem}

{\large %%%%%%%%%%%%%%%%%%%%%%%%%%%%%%%%%%%%%%%%%%%%%%%%%%%%%
%\begin{proof}
%We know that our special building blocks are isomorphic with I. Stevens's buil%ding blocks. Therefore we can apply I. Stevens's Uniqueness Theorem to obtain %an automorphism $f$ with the desired properties in the conclusion of our theor%em except that it is approximately inner.\\ 
%\indent It is a consequence of the I. Stevens's Uniqueness Theorem that the au%tomorphism $f$, obtained above, fixes the spectrum. It is well known, see \cit%e{rae}, Theorem 5.42, p.141, that an automorphism of continuous-trace C*-algeb%ras with spectrum the closed interval $[0,1]$, which fixes the spectrum, is ne%cessarily an approximately inner automorphism.  
%\end{proof}
%%%%%%%%%%%%%%%%%%%%%%%%%%%%%%%%%%%%%%%%%%%%%%%%%%%%%%%%
}

\begin{proof}
{\large Because of the isomorphism theorem 4.13 from \cite{ici}, there is no
loss of generality to assume that our building blocks are in a very special
form }

{\large 
\begin{equation*}
A\cong \left( 
\begin{array}{cccccc}
C_{0}(A_1) & C_{0}(A_1) & C_{0}(A_1) & \dots & C_{0}(A_1) &  \\ 
C_{0}(A_1) & C_{0}(A_2) & C_{0}(A_2) & \dots & C_{0}(A_2) &  \\ 
C_{0}(A_1) & C_{0}(A_2) & C_{0}(A_3) & \dots & C_{0}(A_3) &  \\ 
\vdots & \vdots & \vdots & \ddots &  &  \\ 
C_{0}(A_1) & C_{0}(A_2) & C_{0}(A_3) & \dots & C[0,1] & 
\end{array}
\newline
\right). 
\end{equation*}
}

{\large \indent Notice that the cancellation property holds for the unital
sub-C*-algebra of $A$ and any projection of $A$ is Murray-von Neumann
equivalent to a projection inside of the unital sub-C*-algebra. Therefore
the cancellation property holds for $A$. A similar argument shows that the
cancellation property holds for any continuous-trace C*-algebra with the
spectrum the closed interval $[0,1]$.\newline
\indent Since $\varphi_0 = \psi_0$, we can assume that $\varphi(p)= \psi(p)$%
, where $p$ is the unit of the sub-C*-algebra $C([0,1])$ of $A$. In other
words the restrictions of the maps to the unital subalgebra share the same
unit.\newline
\indent The stable part of the Elliott invariant (i.e., the $K_0$ group and
the affine function space Aff$T^+$) of $A$ and of $C([0,1])$ is the same.
Let us restrict the two maps $\varphi$ and $\psi$ to the unital
sub-C*-algebra $C([0,1])$. The image of $C([0,1])$ under $\varphi$ and $\psi$
is up to a unitary a full matrix algebra over the interval. Then using
assumptions 1 and 2 we notice that the hypotheses of the Elliott Uniqueness
Theorem (\cite{ell2}, Theorem 6), are fulfilled. Hence we get a partial
isometry $V$ of $B$ (a unitary inside of the full matrix sub-C*algebra of $B$%
) such that 
\begin{equation*}
||\varphi(f_{A_i} \otimes e_{nn})-V\psi(f_{A_i} \otimes e_{nn})V^{*}||\leq
\epsilon, \;\;i\in \{1, \dots , n\}.
\end{equation*}
\indent We want this relation to hold for the case when the domain is $A$.
We follow a strategy already present in the case of full matrix over the
interval. An important data that we will use is that the domain algebra $A$
has a finite presentation. In fact we will use the concrete description of
this presentation that was given in \cite{ici}, Section 8. The set of
generators consists of elements of the form $f_{A_i}\otimes e_{in}$ which
are certain positive functions tensor the matrix units. \newline
\indent For each $i$ let $u_i$ be a continuous function defined on $[0,1]$
which is equal to $1$ on $A_i$ except near the end points of each open
subinterval of $A_i$ and $0$ otherwise. One can think of $u_i$ as an
approximate unit of the functions $f_{A_i}$, $i \in \{1,\dots,n\}$ and later
estimates depend on the size of the subset of $A_i$ where $u_i$ is not equal
to $1$.\newline
\indent Define 
\begin{equation*}
\mathcal{V}= \sum_{i=1}^{n}\varphi(u_i\otimes e_{ni})^{*}V\psi(u_i\otimes
e_{ni}).
\end{equation*}
\indent Then 
\begin{equation*}
\mathcal{V}\psi(f_{A_i} \otimes e_{ni})\mathcal{V}^{*}=
\end{equation*}
\begin{equation*}
=(\sum_{k=1}^{n}\varphi(u_k\otimes e_{nk})^{*}V\psi(u_k\otimes
e_{nk}))\psi(f_{A_i} \otimes e_{ni})(\sum_{l=1}^{n}\psi(u_l\otimes
e_{ln})^{*}V^{*}\varphi(u_l\otimes e_{nl}))=
\end{equation*}
\begin{equation*}
=\varphi(u_n\otimes e_{nn})V\psi(f_{A_i}\otimes
e_{ni})(\sum_{l=1}^{n}\psi(u_l\otimes e_{ln})^{*}V^{*}\varphi(u_l\otimes
e_{nl}))=
\end{equation*}
\begin{equation*}
=\varphi(u_n\otimes e_{nn})V\psi(f_{A_i}\otimes
e_{ni})(\sum_{l=1}^{n}\psi(u_l\otimes e_{nl})V^{*}\varphi(u_l\otimes
e_{nl}))=
\end{equation*}
\begin{equation*}
=\varphi(u_n\otimes e_{nn})V\psi(f_{A_i}\otimes e_{ni})\psi(u_i\otimes
e_{in})V^{*}\varphi(u_i \otimes e_{ni})=
\end{equation*}
\begin{equation*}
=\varphi(u_n\otimes e_{nn})V\psi(f_{A_i}\otimes e_{nn})V^{*}\varphi(u_i
\otimes e_{ni}).
\end{equation*}
\indent Now we have that 
\begin{equation*}
\varphi (f_{A_i} \otimes e_{ni})= \varphi(u_n\otimes
e_{nn})\varphi(f_{A_i}\otimes e_{nn}) \varphi(u_i \otimes e_{ni}).
\end{equation*}
\indent Therefore 
\begin{equation*}
||\varphi (f_{A_i} \otimes e_{ni})-\mathcal{V}\psi(f_{A_i} \otimes e_{ni})%
\mathcal{V}^{*}||=
\end{equation*}
\begin{equation*}
||\varphi(u_n\otimes e_{nn})(\varphi(f_{A_i}\otimes
e_{nn})-V\psi(f_{A_i}\otimes e_{nn})V^{*})\varphi(u_i \otimes e_{ni})||\leq
\end{equation*}
\begin{equation*}
\leq ||\varphi(u_n\otimes e_{nn})||\epsilon ||\varphi(u_i \otimes e_{ni})||
\end{equation*}
\indent i.e. it can be made as small as needed.\newline
\indent We want to argue that $\mathcal{V}$ gives rise to a partial
isometry. Let us calculate 
\begin{equation*}
\mathcal{V}^{*}\mathcal{V}=
\end{equation*}
\begin{equation*}
=\sum_{l=1}^{n} \psi(u_l \otimes e_{nl})^{*}V^{*}\varphi(u_l\otimes
e_{nl})\sum_{i=1}^{n}\varphi(u_i\otimes e_{ni})^{*}V\psi(u_i\otimes e_{ni})=
\end{equation*}
\begin{equation*}
=\sum_{l=1}^{n} \psi(u_l \otimes e_{ln})V^{*}\varphi(u_l\otimes
e_{nl})\sum_{i=1}^{n}\varphi(u_i\otimes e_{in})V\psi(u_i\otimes e_{ni})=
\end{equation*}
\indent Assuming that each $u_i$ is equal to $1$ on the open intervals $A_i$
except small neighbourhood around the end points of $A_i$ we get 
\begin{equation*}
=\sum_{i=1}^{n} \psi(u_i \otimes e_{in})V^{*}\varphi(u_i\otimes
e_{nn})V\psi(u_i\otimes e_{ni})
\end{equation*}
which is very close to 
\begin{equation*}
\sum_{i=1}^{n} \psi(u_i \otimes e_{in})\psi(u_l\otimes
e_{nn})\psi(u_i\otimes e_{ni})=
\end{equation*}
\begin{equation*}
=\sum_{i=1}^{n} \psi(u_i \otimes e_{ii})
\end{equation*}
which is the value of the projection-valued map of the hereditary
sub-C*-algebra generated by $\psi(A)$ inside $B$. In other words $\mathcal{V}%
^{*}\mathcal{V}$ is as close as we want to be a projection. It is important
to notice that this is true if we are not in a small neighbourhood of the
singularity points of the dimension function of the hereditary
sub-C*-algebra generated by $\psi(A)$ (i.e. whenever $u_i=1$). \newline
\indent Similarly $\mathcal{V}\mathcal{V}^{*}$ is almost equal to the $%
\sum_{i=1}^{n} \varphi(u_i \otimes e_{ii})$ if we are not in a small
neighbourhood of the singularity points of the dimension function of the
hereditary sub-C*-algebra generated by $\varphi(A)$. Notice that any
singularity point $y_0$ of the dimension function of the hereditary
sub-C*-algebra generated by $\varphi(A)$ or $\psi(A)$ has the property that
there is an eigenfunction $\lambda_i$ such that $\lambda_i(y_0)$ is a
singularity point of the dimension function $d_A$ of $A$. In addition $%
\lambda_i$ is uniform continuous function from $[0,1]$ to $[0,1]$. Hence
small neighbourhoods of $y_0$ correspond to small neighbourhoods of some
singularity point of $d_A$.\newline
\indent From the polar decomposition $\mathcal{V}= \mathcal{W}|\mathcal{V}|$
we get a partial isometry $\mathcal{W}$. We claim that $\mathcal{W}$ still
intertwines approximately the two maps $\varphi$ and $\psi$, i.e., 
\begin{equation*}
||\varphi (f_{A_i} \otimes e_{ni})-\mathcal{W}\psi(f_{A_i} \otimes e_{ni})%
\mathcal{W}^{*}|| < 3 \epsilon,
\end{equation*}
\begin{equation*}
||\mathcal{W}^{*}\varphi (f_{A_i} \otimes e_{ni})\mathcal{W}-\psi(f_{A_i}
\otimes e_{ni})|| < 3 \epsilon.
\end{equation*}
}

{\large \indent This is true because 
\begin{equation*}
||\varphi (f_{A_{i}}\otimes e_{ni})-\mathcal{W}\psi (f_{A_{i}}\otimes e_{ni})%
\mathcal{W}^{\ast }||=
\end{equation*}%
\begin{equation*}
=||\varphi (f_{A_{i}}\otimes e_{ni})-\mathcal{V}\psi (f_{A_{i}}\otimes
e_{ni})\mathcal{V}^{\ast }+\mathcal{W}|\mathcal{V}|\psi (f_{A_{i}}\otimes
e_{ni})\mathcal{|V|}\mathcal{W}^{\ast }-\mathcal{W}\psi (f_{A_{i}}\otimes
e_{ni})\mathcal{W}^{\ast }||\leq 
\end{equation*}%
\begin{equation*}
\leq ||\varphi (f_{A_{i}}\otimes e_{ni})-\mathcal{V}\psi (f_{A_{i}}\otimes
e_{ni})\mathcal{V}^{\ast }||+||\mathcal{W}\mathcal{|V|}\psi
(f_{A_{i}}\otimes e_{ni})\mathcal{|V|}\mathcal{W}^{\ast }-\mathcal{W}\psi
(f_{A_{i}}\otimes e_{ni})\mathcal{W}^{\ast }||\leq 
\end{equation*}%
\begin{equation*}
\leq \epsilon +||\mathcal{|V|}\psi (f_{A_{i}}\otimes e_{ni})\mathcal{|V|}%
-\psi (f_{A_{i}}\otimes e_{ni})||\leq 
\end{equation*}%
\begin{equation*}
\leq \epsilon +||\mathcal{|V|}\psi (f_{A_{i}}\otimes e_{ni})\mathcal{|V|}-%
\mathcal{|V|}\psi (f_{A_{i}}\otimes e_{ni})+\mathcal{|V|}\psi
(f_{A_{i}}\otimes e_{ni})-\psi (f_{A_{i}}\otimes e_{ni})||\leq 
\end{equation*}%
\begin{equation*}
\leq \epsilon +\epsilon +\epsilon =3\epsilon .
\end{equation*}%
\indent and similarly we get the other desired inequality.\newline
\indent Hence we have constructed a family of partial isometries $\mathcal{W}
$ from the hereditary sub-C*-algebra generated by $\varphi (A)$ to the
hereditary sub-C*-algebra generated by $\psi (A)$. In addition $\mathcal{W}$
induces an isomorphism between the two above mentioned hereditary
sub-C*-algebras. In particular it implies that the two hereditary
sub-C*-algebras have the same dimension function. \newline
\indent Next we will show how to approximate $\mathcal{W}$ with a unitary in
the unitazation of the codomain algebra.\newline
\indent Let us start by applying Theorem 4.12 of \cite{ici} to the
projection-valued function corresponding to the hereditary sub-C*-algebra
generated by $\varphi (A)$. Hence we get a decomposition, possibly infinite,
in terms of functions each of which is projection-valued of rank $1$ on a
certain open subset of $[0,1]$ and zero otherwise. Notice that the
discontinuity points of the dimension function of the hereditary
sub-C*-algebra generated by $\varphi (A)$ correspond to the discontinuity
points of the functions appearing in the decomposition and the open sets are
increasing in a suitable sense. \newline
\indent Next we apply Lemma 6.2 for each point at singularity in the
interval $[0,1],$ or, in other words, to each function appearing on  the
decomposition. Thus, we have a family of unitaries that preserves the
continuity of the continuous elements of the hereditary sub-C*-algebra $%
\varphi (A)$ and at the same time has the property that it still intertwines
the two maps.\newline
}
\end{proof}

{\large \indent In the following lemma the hereditary sub-C*-algebras $H_1$ and $H_2$ are assumed 
to be
continuous bundles over $[0,1]$ (for more details about continuous bundles
of C*-algebras see \cite{was}).\newline
\indent If $A$ is a continuous bundle of C*-algebras over $[0,1]$ then $A^t$
stands for the fiber of $A$ over $t$. }

\begin{lemma}
{\large Let $H_1$ and $H_2$ be hereditary sub-C*-algebra of $M_2(C[0,1])$
with the same spectrum $[0,1]$ and identical dimension function equal to $1$
on the closed interval $[0,t_0]$ and equal to $2$ on the half-open interval $%
(t_0,1]$, $t_0 \in (0,1)$. Let $\mathcal{W}=(W(t))_{t\in [0,1]}$ be a family
of partial isometries indexed by the points of $[0,1]$. For each $t \in [0,1]
$, $W_t:M_2(\mathbb{C})\rightarrow M_2(\mathbb{C})$ such that $W(t)W(t)^*=$
the unit of $H_1^t$ and $W(t)^*W(t)=$ the unit of $H_2^t$. Then there exists
a family $\mathcal{W^{\perp}}$ of partial isometries indexed by $[0,1]$ such
that $\mathcal{W}+\mathcal{W^{\perp}}$ is a unitary inside of $M_2(C[0,1])$
and $(W+W^{\perp})_{t}(f)(t)=W_t(f)(t)$ for any continuous function $f \in
H_1$ and $t \in [0,1]$. }
\end{lemma}

\begin{proof}
{\large Diagrammatically the dimension function of $H_1$ and $H_2$ can be
pictured as follows. }

\begin{center}
\begin{picture}(400,200)(-50,0)
\put(145,57){]}
\put(145,107){(}
\put(50,60){\line(1,0){95}}
\put(145,110){\line(1,0){95}}
\put(50,10){\vector(1,0){190}}
\put(50,10){\vector(0,1){185}}
\end{picture}

Figure 5.  Dimension function of $H_1$ and $H_2$.
\end{center}

%\begin{center}
%{\large \includegraphics [height = 3 cm] {dim.eps}\newline
%Figure 5. Dimension function of $H_1$ and $H_2$ }
%\end{center}

{\large \indent We construct the family $\mathcal{W}^{\perp}=(\mathcal{W}%
_{t}^{\perp})_{t\in [0,1]}$ as follows. Fix a $t$ in $[0,1]$, $t \leq t_0$. $%
W(t)$ is a partial isometry on some dimension-one subspace of $M_2(\mathbb{C}%
)$. Hence $W_t(M)=c(M)M_t$ where $c(M)$ is a constant depending on $M$ and $%
M_t$ is a projection matrix in $M_2(\mathbb{C})$. Let $W_t^{%
\perp}=c(M)(I_2-M_t)$. Notice that $W_t +W_t^{\perp}$ is a unitary operator
on $M_2(\mathbb{C})$. If $t >t_0$ then $W_t^{\perp}=0$.\newline
\indent The family of unitaries $(W_t+W_t^{\perp})_{t\in [0,1]}$ is
continuous except at the point $t_0$. Our work below shows that this family
can be modified to be continuous overall $[0,1]$.\newline
\indent Extend $(W_t)_{t\in [0,t_0]}$ to be a continuous family $%
(W_{t}^{1})_{t\in [0,1]}$ of partial isometries on dimension-one subspaces
of $M_2(\mathbb{C})$. $W_{t_0}^{\perp}$ and $\lim\limits_{t\rightarrow
t_0,t>t_0}(W_t-W_t^{1})$ are two partial isometries on the same dimension
one subspace of $M_2(\mathbb{C})$, hence they differ by a constant of
absolute value one, i.e. 
\begin{equation*}
W_{t_0}^{\perp}=c \lim\limits_{t\rightarrow t_0,t>t_0}(W_t-W_t^{1}).
\end{equation*}
\indent Define the continuous family of unitaries $(U_t)_{t\in [0,1]}$ to be 
$U_t=W_t +W_t^{\perp}$ if $t\leq t_0$ and $U_t=W_t^{1} + c(W_t - W_t^{1})$
if $t > t_0$. }

{\large %%%%%%%%%%%%%%%%%%%%%%%%%%%%%%%%%%%%%%%%%%%%%%%%
% inside the unitization of $A$, where $\mathcal{W}^{\perp}$ is a partial isome%tric multiplier on the ortogonal complement of the image of $\mathcal{W}$. Thi%s family of unitaries is continuous if we are away from the singularity point %but not continuous at the singularity point.\\
%%%%%%%%%%%%%%%
}

{\large \indent Hence the continuous family of unitaries $W_t$ is given by $%
U_t$ and $(U_{t}(f)(t)=W_t(f)(t)$ for any continuous function $f \in H_1$
and $t \in [0,1]$.\newline
}

{\large %%%%%%%%%%%%%%%%
%\indent It is important to notice that any continuous element is zero at the d%iscontinuity point after we subtract from it the cut down of the element by a %certain rank one projection. The certain rank one projection that is refered t%o is a projection equal to the projection-valued map of $H$ on the interval $[%0,t_0]$ and is a continuous extension on $(t_0,1]$ smaller than the projection%-valued map of $H$. Then the unitary family $\mathcal{W}+\mathcal{W}^{\perp}$ %evaluated on the elements of $F$, on the neighbourhood of the discontinuity po%int which corresponds to the larger value of the dimension function, namely $2%$, is very small.  Therefore any change of the unitary family will be small on% finitely many elements. We change the family of unitaries so that it matches %the family of unitaries from the other side of the discontinuity point. This i%s possible because there is a path of unitaries from $(\mathcal{W}+\mathcal{W}%^{\perp})_{t_1<t_0}$ to $(\mathcal{W}+\mathcal{W}^{\perp})_{t_2>t_0}$, where $%t_1$ and $t_2$ are chosen sufficiently close to $t_0$. In this manner we obtai%n a continuous family of unitaries which evaluated on finitely many elements i%s approximately equal to the evaluation of the given partially isometric multi%plier $\mathcal{W}$ evaluated on the same finite set of elements $F$.
%%%%%%%%%%%%
}
\end{proof}

\section{\protect\large Inductive limits of special continuous trace
C*-algebras}

{\large \indent Next let us show that the Existence Theorem and the Uniqueness
Theorem presented above can be applied, i.e., that the hypotheses of the theorems
can be fulfilled. As a first step in this direction let us  show that an
inductive limit of continuous-trace C*-algebras with spectrum $[0,1]$ 
(or disjoint unions of closed intervals)  is
isomorphic to an inductive limit of special continuous-trace C*-algebras. 
\newline
\indent The basic tools in establishing this step are the fact that special
continous trace C*-algebras are semiprojective (cf. \cite{ici}, Theorem 6.5)
and a result by T. Loring (\cite{lor}, Lemma 15.2.2) which for the convenience 
of the reader we state below: \\ 
\indent Suppose that $A$ is a C*-algebra
containing a (not necessarily nested) sequence of sub-C*-algebras $A_n$ with
the property that for all $\epsilon > 0$ and for any finite number of elements $%
x_1, \dots,x_k$ of $A$, there exist an integer $n$ such that
\begin{equation*}
\{ x_1, \dots, x_k \} \subset_{\epsilon}A_n.
\end{equation*}
\indent If each $A_n$ is weakly semiprojective and finitely presented, then 
\begin{equation*}
A \cong \lim\limits_{\rightarrow}(A_{n_k}, \gamma_{k})
\end{equation*}
for some subsequence of $(A_n)$ and some maps $\gamma_k:A_{n_k}\rightarrow A_{n_{k+1}}$%
. }

\begin{prop}
{\large Let $A$ be a simple inductive limit of continuous-trace C*-algebras
whose building blocks have their spectrum homeomorphic to $[0,1]$. Then $A$
is an inductive limit of direct sums of special continuous-trace C*-algebras 
with spectrum $[0,1]$. }
\end{prop}

\begin{proof}
{\large \indent In Proposition 5.4 and Theorem 6.5 of \cite{ici} it is
proved that the class of special continuous trace C*-algebras with spectrum $%
[0,1]$ are finitely presented and have weakly stable relations. Each
building block from the inductive limit decomposition of $A$ can be
approximated by special continuous trace C*-algebras (cf. Theorem 6.14 of \cite{ici}).
Then $A$ satisfies Loring's hypothesis where the sequence of semiprojective
algebras is given by the special algebras from the approximation of the
building blocks. Thus the Loring's lemma implies that $A$ is an inductive
limit of special continuous trace C*-algebras. }
\end{proof}

\section{\protect\large Getting a non-zero gap at the level of affine
function spaces}

{\large \indent To be able to exactly intertwine the non-stable part of the
invariant it is useful to know that the dimension function of any building
block $A_m$ or $B_m$ is taken by the homomorphism $\phi_{m,m+1}$
respectively $\psi_{m,m+1}$ into a function smaller than or equal to the
dimension function of $A_{m+1}$ or $B_{m+1}$ such that a non-zero gap
arises. In other words we want to exclude the possible cases when the
dimension function is taken into the next stage dimension function such that
equality holds at a point or at more points. We shall show this in the
following lemma. Recall that because of Proposition 7.1, the algebras that
we want to classify can be assumed to be inductive limits of special
continuous trace C*-algebras with spectrum $[0,1]$, i.e., $A \cong
\lim\limits_{\rightarrow}(A_{n},\phi_{nm})$ and $B \cong
\lim\limits_{\rightarrow}(B_{n},\psi_{nm})$, where $A_n,B_n$ are special
continuous trace C*-algebras. }

\begin{lemma}
{\large Let $A= \lim\limits_{\rightarrow}(A_n,\phi_{nm})$ be a simple
C*-algebra, where each $A_n$ is a special continuous trace C*-algebra with
spectrum the closed interval $[0,1]$ and the dimension function assumed to
be a finite-valued bounded function. Then there exists $\delta_1 >0$, a
subsequence $(A_{n_{i}})_{n_{i}\geq 0}$ of $(A_n)_n$ and a sequence of maps $%
\phi_i:A_{n_i}\rightarrow A_{n_{i+1}}$ such that \newline
\indent 1. $A \cong \lim\limits_{\rightarrow}(A_{n_i},\phi_{n_im_i})$,\newline
\indent 2. ${(\phi_{n_1n_2})}_T(\hat{P}_{A_{n_1}}) + \delta_1 < \hat{P}%
_{A_{n_2}}$,\newline
where the inequality holds pointwise, ${(\phi_{nm})}_T$ is the induced map at
the level of the affine function spaces, $P_{A_{n_1}}$ and $P_{A_{n_2}}$ are 
the units of the biduals of $A_{n_1}$ and $A_{n_2}$, and $\hat{P}_{A_{n_1}}$ 
and $\hat{P}_{A_{n_2}}$ denote the corresponding lower semicontinuous functions.}
\end{lemma}

\begin{proof}
{\large Let $A$ be equal to $\lim\limits_{\rightarrow}A_n$ with maps $%
\phi_{n,m}:A_n \rightarrow A_m$.\newline
\indent The plan is to keep the same building blocks and to change slightly
the maps with respect to some given finite sets such that the desired
property holds. To do this we use the property that the building blocks that
appear in the inductive limit decomposition are weakly semiprojective.%
\newline
\indent Assume that the dimension function of $\phi_{12}(A_1)$ equals the
dimension function of $A_2$ at some point or even everywhere and let $%
\epsilon >0$, $F_1 \subset A_1$ be given. Because the largest value of the
dimension function of the hereditary sub-C*-algebra generated by $%
\phi_{12}(A_1)$ inside $A_2$ is attained on an open subset $U$ of $[0,1]$,
let us construct another dimension function as follows: shrink one of the
open intervals of the open set $U$ to get $U^{\prime}$ and in exchange
enlarge the interval adjacent to that discontinuity point. $U^{\prime}$ is
constructed in a such a way that is as close as necessary to the given $U$. }

{\large %as in the diagrammatic representation below
}

{\large %xfig
}

{\large \indent In this manner we find a sub-C*-algebra $B$ which is as
close as we want to the hereditary sub-C*-algebra generated by $%
\phi_{12}(A_1)$ inside of $A_2$. Next we use that $A_1$ is weakly
semiprojective to find another *-homomorphism $\rho_1: A_1 \rightarrow B$
which is close within the given $\epsilon$ on the given finite set $F_1$. }

{\large %%%%%%%%%%%%%%%%%%%%%%%%%%%%%%%%%%
%such that each continuous entry function $f_{ij}$, where $(f_{ij})_{ij} = f \i%n \phi_{12}(F_1)$, unchanged on $U'$ and zero otherwise, differs in norm from %the same function $f=(f_{ij})_{ij}$ defined on $U$ by at most $\epsilon$. From% $B$ to $A_2$ we have a canonical *-homomorphism $\psi_1$. We want to construc%t a *-homomorphism $\phi_1 : A_1 \rightarrow A_2$ such that $||\psi_1 \circ \p%hi_1 (g)-g|| < \epsilon$ for $g \in \phi_{12}(F_1)$. To construct $\phi_1$ it %is sufficient to construct an eigenvalue pattern $T$ that is close to the iden%tity and it is compatible with the dimension functions of $B$ and $A_2$, call %them $d_B$ and $d_{A_2}$ respectively, in the sense that $T(d_B) \leq d_{A_2}$%.\\ 
%\indent Define $T$ to be a continuous function from $[0,1]$ to $[0,1]$ which i%s necessarily identity function on $U'$ as well as on complement of $U$ and ze%ro on the points which are in $U$ but not in $U'$. The set $U'$ was chosen suc%h that the continuous functions have variation smaller than $\epsilon$ on the %set of points which are in $U$ but not in $U'$.  Hence $T$ has the desired pro%perty.\\ 
%%%%%%%%%%%%%%%%%%%%%%%%%%%%%%%%%%%
}

{\large %%%%%%%%%%%%%%%%%%%%
% Approximate $\phi_{n,m}(A_n)$ by special building blocks. Notice that this ti%me it is an approximation of special building blocks by special building block%s. This is possible by approximating one of the open intervals where the dimen%sion function of $A_m$ takes its largest value by a smaller open interval from% inside.\\
%%%%%%%%%%%%%%%%%%%%
\indent Then there exists some open interval between the dimension function
of $A_2$ and the dimension function of the $B$. This open interval
corresponds to a non-zero ideal $I_1$ inside of $A_2$. Now the image of $I_1$
in the inductive limit is also a non-zero ideal. Since the inductive limit
is simple, it implies that the ideal is the whole algebra. We know that
there are full projections in the inductive limit. Therefore there is a
finite stage in the inductive limit of the ideals coming from $I_1$ that has
a full projection. Assume that the finite stage is inside of $A_k$. This
means that at that stage the image of the ideal $I_1$ is $A_k$. Pick a
strictly positive element $a_1$ in $I_1$. Then the image of $a_1$ in $A_k$
will be strictly positive at each point from $[0,1]$, $k > 1$. This shows
that the image of the dimension function $d_{B}$ inside the dimension of $A_k
$ has a gap of at least $1$ everywhere in $[0,1]$.\newline
\indent Because of the normalizations of the affine function, this gap of
size 1 will correspond to some strictly non-zero $\delta_1$. To complete the
proof we relabel $B$ as $A_{n_1}$, $A_k$ as $A_{n_2}$ etc. }

{\large %%%%%%%%%%%%%%%%%%%%%%%%%%%%%%%
% At the next stage, $A_{k+1}$, the gap will be at least $M_{n,k+1} c_n$ for so%me $c_n$ a small strictly positive constant smaller than $1$. This shows that %a gap arises if one changes the image $\phi_{n,m}(A_n)$ with a smaller special% building blocks. Notice that using stability of the relations of $A_n$ there %is a map $\phi_{n,m}'$ from $A_n$ to the algebra that approximates $\phi_{n,m}%(A_n)$. This new map is close to the given map $\phi_{n,m}$. This shows that t%he inductive limit of the old sequence of building blocks with these new maps %is still isomorphic with the given inductive limit $A$. Hence $A_{m}'$ are equ%al to $A_m$ and it is only that the maps differ.
%%%%%%%%%%%%%%%%%%%%%%%% 
}
\end{proof}

\begin{corollary}
{\large Let $A= \lim\limits_{\rightarrow}(A_n,\phi_{n,m})$ be a simple
C*-algebra. Then there exists a sequence $(\delta_i)_{i\geq 1}$, $\delta_i >0
$, a subsequence of algebras $(A_{n_i})_{i \geq 1}$ of $(A_i)_{i \geq 1}$
and a sequence of maps $\phi: A_{n_i} \rightarrow A_{n_{i+1}}$ such that:%
\newline
\indent 1. $A \cong \lim\limits_{\rightarrow}(A_{n_i},\phi_{n_i,m_i})$,%
\newline
\indent 2. $\phi_{T_{n_i,n_{i+1}}}(\hat{P}_{A_{n_i}}) + \delta_i < \hat{P}%
_{A_{n_{i+1}}}$.\newline
}
\end{corollary}

\begin{proof}
{\large Follows by successively applying the previous lemma. }
\end{proof}

\section{\protect\large Pulling back of the isomorphism between inductive
limits at the level of the invariant}

{\large \noindent \indent {\bf Step 1 The intertwining between the stable
part of the invariant}\newline
\indent With no loss of generality we assume that the building blocks have
the following concrete representation 
\begin{equation*}
\left( 
\begin{array}{cccccc}
C_{0}(A_1) & C_{0}(A_1) & C_{0}(A_1) & \dots & C_{0}(A_1) &  \\ 
C_{0}(A_1) & C_{0}(A_2) & C_{0}(A_2) & \dots & C_{0}(A_2) &  \\ 
C_{0}(A_1) & C_{0}(A_2) & C_{0}(A_3) & \dots & C_{0}(A_3) &  \\ 
\vdots & \vdots & \vdots & \ddots &  &  \\ 
C_{0}(A_1) & C_{0}(A_2) & C_{0}(A_3) & \dots & C[0,1] & 
\end{array}
\newline
\right). 
\end{equation*}
\indent One can distinguish a full unital hereditary sub-C*-algebra }

{\large 
\begin{equation*}
\left( 
\begin{array}{cccccc}
0 & 0 & 0 & \dots & 0 &  \\ 
0 & \ddots & 0 & \dots & 0 &  \\ 
0 & \dots & C[0,1] & \dots & C[0,1] &  \\ 
\vdots & \vdots & \vdots & \ddots &  &  \\ 
0 & \dots & C[0,1] & \dots & C[0,1] & 
\end{array}
\newline
\right). 
\end{equation*}
}

{\large \indent The unital hereditary sub-C*-algebra has the same stable
invariant (i.e., $K_0$, AffT$^+$ and the pairing) as the given C*-algebra.
Moreover the unital hereditary sub-C*-algebra is a full matrix algebra over
the closed interval $[0,1]$. Using this fact we derive an intertwining
between the stable invariant, as is shown in \cite{ste0} or originally in %
\cite{ell2}.\newline
\indent It is important to mention the method of normalizing the affine
function spaces. Pick a full projection $p_1 \in A_1$. Normalize the affine
space Aff$T^{+}A_1$ with respect to $p_1$. Next consider a image of $p_1$ in 
$A_2$ under the map at the dimension range level, call it $p_2$. Normalize
Aff$T^{+}A_2$ with respect to $p_2$. Note that the map which is induced at
the affine level is a contraction. Continue in this way so that we obtain an
inductive limit sequence at the level of the affine spaces, with all the
maps being contractions: 
\begin{equation*}
\mathrm{Aff}T^{+}A_1 \rightarrow \mathrm{Aff}T^{+}A_2 \rightarrow \dots
\rightarrow \mathrm{Aff}T^{+}A.
\end{equation*}
\indent Let $p_{\infty}$ denote the image of $p_1$ in the inductive limit $A$
and denote by $q_{\infty}$ a representative of $\phi_0(p_{\infty})$ in $B$.
Then there exists $q_1 \in B_1$ such that the image of $q_1$ is $q_{\infty}$
in the inductive limit. Normalize the Aff$T^{+}B_1$ with respect to $q_1$,
Aff$T^{+}B_2$ with respect to a image of $q_1$ in $B_2$ and so on. Hence we
obtain another inductive limit of affine spaces with contractions maps 
\begin{equation*}
\mathrm{Aff}T^{+}A_1 \rightarrow \mathrm{Aff}T^{+}A_2 \rightarrow \dots
\rightarrow \mathrm{Aff}T^{+}A
\end{equation*}
\begin{equation*}
\mathrm{Aff}T^{+}B_1 \rightarrow \mathrm{Aff}T^{+}B_2 \dots \rightarrow 
\mathrm{Aff}T^{+}B
\end{equation*}
\indent As already mentioned above, we pull back the invariant for the
unital hereditary sub-C*-algebras (i.e. full matrix algebras or the stable
invariant). This will give rise to an exact commuting diagram at the $K_0$%
-level, an approximate commuting diagram at the affine function spaces level
and an exact pairing. The compatibility can be made exact as shown in \cite%
{ell3} by noticing that , because of simplicity, non-zero positive elements
in both $K_0$ and Aff$T^+$ are sent into strictly positive elements and then
normalize the affine function spaces in a suitable way.\newline
\indent To summarize, we now have a commutative diagram 
\begin{equation*}
\begin{array}{ccccccc}
C[0,1] & \overset{\phi_{12}}{\longrightarrow} & C[0,1] & \overset{\phi_{23}}{%
\longrightarrow} & \dots & \longrightarrow & (\mathrm{AffT}^{+}A, \mathrm{%
Aff^{\prime}A}) \\ 
\downarrow \tau_1 & \nearrow \tau^{\prime}_1 & \downarrow \tau_2 & \nearrow
\tau^{\prime}_2 & \nearrow &  & \updownarrow \\ 
C[0,1] & \overset{\psi_{12}}{\longrightarrow} & C[0,1] & \overset{\psi_{23}}{%
\longrightarrow} & \dots &  & (\mathrm{AffT}^{+}B, \mathrm{Aff^{\prime}B})%
\end{array}
\end{equation*}
where Aff$T^+A_i$ and Aff$T^+B_i$ are identified with $C([0,1])$ and each
finite stage algebra $A_i$ and $B_i$ is assumed to have only one direct
summand.\newline
\indent For us it is very important to study the pulling back of the
non-stable part of the invariant.\newline
\indent {\bf Step 2. The intertwining of the non-stable part of the
invariant} \indent As I. Stevens mentioned in \cite{ste0}, at this moment we
know that the non-stable part of the invariant is only approximately mapped
at a later stage into the non-stable part of the invariant. \newline
\indent To be able to apply the Existence Theorem 5.1, one needs to check
that hypothesis 2 can be ensured. Otherwise, a counterexample can be given
to the Existence Theorem, as shown in Section 5.1 above. The special
assumption from the hypothesis of the isomorphism theorem, $\phi_T(\mathrm{%
Aff^{\prime}}A) \subseteq \mathrm{Aff^{\prime}}B$, as well as Corollary 8.2
will be used to prove the above mentioned claim.\newline
\indent By applying Corollary 8.2 to the given inductive limits $A=
\lim\limits_{\rightarrow}(A_n,\phi_{n,m})$, $B=
\lim\limits_{\rightarrow}(B_n,\phi_{n,m})$ we get two sequences $%
(\delta_{i})_{i \geq 1}, \delta_i >0$ and $(\delta^{\prime}_{i})_{i \geq 1},
\delta^{\prime}_i >0$ respectively, and two subsequences of algebras such
that after relabeling, we can assume that $\phi_{ii+1}(\hat{P}_{A_{i}})+
\delta_i < \hat{P}_{A_{i+1}}, \psi_{ii+1}(\hat{P}_{A_{i}})+ \delta_i < \hat{P%
}_{A_{i+1}}, \psi_{ii+1}(\hat{P}_{A_{i}})+ \delta_i < \hat{P}_{A_{i+1}}$ and 
$\psi_{ii+1}(\hat{P}_{A_{i}})+ \delta_i < \hat{P}_{A_{i+1}}$ for all $i \geq
1.$\newline
\indent Reworking the intertwining of the stable invariant for the new
sequences of algebras and the new maps that have gaps $\delta_i$ we obtain
the following intertwining 
\begin{equation*}
\begin{array}{ccccccc}
C([0,1]) & \overset{\phi_{12}}{\longrightarrow} & C([0,1]) & \overset{%
\phi_{23}}{\longrightarrow} & \dots & \longrightarrow & (\mathrm{AffT}^{+}A, 
\mathrm{Aff^{\prime}A}) \\ 
\downarrow \tau_1 & \nearrow \tau^{\prime}_1 & \downarrow \tau_2 & \nearrow
\tau^{\prime}_2 & \nearrow &  & \updownarrow \\ 
C[0,1] & \overset{\psi_{12}}{\longrightarrow} & C[0,1] & \overset{\psi_{23}}{%
\longrightarrow} & \dots &  & (\mathrm{AffT}^{+}B, \mathrm{Aff^{\prime}B})%
\end{array}%
\end{equation*}
\indent As a consequence of the Thomsen-Li theorem, which in the present
case states that the closed convex hull of the set of all unital
*-homomorphisms of $C([0,1])$ in the strong operator topology is exactly the
set of positive of unital operators on $C([0,1])$, we can assume that all
the maps $\phi_{ii+1}, \psi_{ii+1}, \tau_i, \tau^{\prime}_i$ are given by
eigenvalue patterns. Because each such map takes the unit, say $\hat{p}$,
into the unit, $\widehat{K(p)}$, it follows that each map is an average of
the eigenvalues, i.e., $\phi_{i,i+1}(f)= \sum_{i=1}^{N_i}\frac{f \circ
\lambda_i}{N_i}$, etc.\newline
\indent Let $\hat{P}_{A_1}$ be the image in the affine function space of the
unit in the bidual of $A_1$. Take a continuous function $f$ smaller than $%
\hat{P}_{A_1}$. It is important to say that there are no extra conditions on 
$f$, i.e., $f$ can be any element of the special set AffT$^{\prime}A_1$.
Then there exists $\delta_1>0$ such that 
\begin{equation*}
\phi_{12}(\hat{P}_{A_1})+\delta_1 < \hat{P}_{A_2}.
\end{equation*}
\indent Since $\phi_{12}(f) \leq \phi_{12}(\hat{P}_{A_1})$ we have 
\begin{equation*}
\phi_{12}(f + \delta_1) \leq \phi_{12}(\hat{P}_{A_1}+\delta_1) < \hat{P}%
_{A_2}.
\end{equation*}
\indent Since $\phi_T(\mathrm{Aff^{\prime}}A) \subseteq \mathrm{Aff^{\prime}}%
B$, it follows that there exists a large $N$ and $\epsilon_N \leq \delta_1$
such that 
\begin{equation*}
\tau_{N} \circ \phi_{N-2N-1}\circ\dots \phi_{12}(f+\delta_1)< \hat{P}_{B_N}
+ \epsilon_{N}.
\end{equation*}
\newline
\indent It is important to say that a different choice for $f$ will give
rise to possibly different $N$. This is not a difficulty because we can
always pass to subsequence. Equivalently we have 
\begin{equation*}
\tau_{N} \circ \phi_{N-2N-1}\circ\dots \phi_{12}(f)+\delta_1< \hat{P}_{B_N}
+ \epsilon_{N}.
\end{equation*}
\indent Using $\delta_1 \geq \epsilon_N$ we conclude 
\begin{equation*}
\tau_{N}\circ \phi_{N-2N-1}\circ\dots \phi_{12}(f)< \hat{P}_{B_N},
\end{equation*}
which is the desired strict inequality from the hypothesis 2 of the
Existence Theorem 5.1. }

{\large %%%%%%%%%%%%%%%%%%%%%%%%%%%%
%\indent As shown by I. Stevens in \cite{ste0}, by factoring through a finite d%irect sum of copies of $C(\{.\})$, approximately, we get an intertwining betwe%en the affine function spaces. This is done with respect to the finite sets: $%F'_i \subset F_{i} \subset A_i$, such that $( \bigcup F'_i)^{-}=(A)_{1}^{+}$, %$F_i \subset F_{i+1}, (\bigcup F_i)^{-}=A$, and $G'_i \subset G_{i} \subset B_%i$, such that $( \bigcup G'_i)^{-}=(A)_{1}^{+}$, $G_i \subset G_{i+1}, (\bigcu%p G_i)^{-}=B$. Therefore at the level of the affine space, after we pull back %the isomorphism between the inductive limits to the finite stages, we get an a%pproximate intertwining between affine spaces which approximately preserves th%e scale.\\
%\begin{lemma} Let $F$ be a finite set in $C([0,1])$ and $\epsilon >0$. There a%re $\varphi, \psi$ maps, and $n \in \mathbb{N}$, such that $\varphi:C([0,1]) \%rightarrow \bigoplus _{i=1}^{n} \mathbb{R}$ and $\psi: \bigoplus _{i=1}^{n}\ma%thbb{R} \rightarrow C([0,1])$, continuous, linear, positive such that $$||(\ps%i \circ \varphi-id)(a)|| < ||a|| \epsilon,\;\mathrm{for}\;\mathrm{all}\;a\in F%.$$
%\end{lemma}
%
%\begin{proof} A proof of this lemma can be found in \cite{ste0}, Proposition 2%8.
%\end{proof}
%\indent Using lemma 10.4, we get:
%%%%%%%%%%%%%%%%%%%%%%%%%%%%%%%%%%%%%%%%%%%%%%%%
}

{\large %%%%%%%%%%%%%%%%%%
}

\section{\protect\large The Isomorphism Theorem}

{\large \noindent \indent To complete the proof of the Isomorphism Theorem
3.1 for the algebras $\lim\limits_{\rightarrow} A_i=A$ and $%
\lim\limits_{\rightarrow}B_i=B$, we have to construct an approximate
commutative diagram at the algebra level in the following sense, as was
defined by Elliott in \cite{ell3},\textit{``for any fixed element in any $A_i
$ (or $B_i$), the difference of the images of this element along two
different paths in the diagram, starting at $A_i$ (or $B_i$) and ending at
the same place, converges to zero as the number of steps for which the two
paths coincide, starting at the beginning, tends to infinity.''}\newline
}

{\large \indent At this stage because of Step 2 of the previous section,
Section 9, we can apply the Existence Theorem to generate a sequence of
algebra homomorphisms $\nu_1, \nu_2,\dots$ and $\nu^{\prime}_1,
\nu^{\prime}_2, \dots$ such that $\frac{||\tau_{i}(f)-\nu_{i*}(f)||}{||f||}
\leq \frac{\epsilon}{2^i}$ and $\frac{||\tau^{\prime}_{i}(f)-\nu^{%
\prime}_{i*}(f)||}{||f||} \leq \frac{\epsilon}{2^i}$ for $f\in F_i$ and $%
g\in G_i$, where $\nu_{i*},\nu^{\prime}_{i*},$ are the induced afine maps by
algebra maps $\nu_i,\nu_i^{\prime}$, and $F_i$ and $G_i$ are finite sets.%
\newline
\indent  After relabeling the indices of the inductive limit systems we now
have a (not necessarily approximately commutative) diagram of algebra
homomorphisms 
\begin{equation*}
\begin{array}{ccccccc}
A_1 & \overset{\phi_{12}}{\longrightarrow} & A_2 & \overset{\phi_{23}}{%
\longrightarrow} & \dots & \longrightarrow & A \\ 
\downarrow \tau_1 & \nearrow \tau^{\prime}_1 & \downarrow \tau_2 & \nearrow
\tau^{\prime}_2 &  &  &  \\ 
B_1 & \overset{\psi_{12}}{\longrightarrow} & B_2 & \overset{\psi_{23}}{%
\longrightarrow} & \dots & \longrightarrow & B%
\end{array}
\end{equation*}
that induces an approximately commutative diagram at the level of the
invariant.\newline
This will be done with respect to given arbitrary finite sets $F_i \subset
A_i$ and $G_i \subset B_i$.\newline
}

{\large %%%%%%%%%%%%%%%%%%%%%%%%%%%%%%%%%%%%%%%%%%%%%%%%%%%%%%%%%
\indent To make the diagram approximately commuting we modify the diagonal
maps by composing with approximately inner automorphisms and this will be
done with respect to a given arbitrary finite sets $F_i \subset A_i$ and $%
G_i \subset B_i$ with dense union in $A$ and $B$ respectively.\newline
\indent Here we notice that we can apply the Uniqueness Theorem to the data
obtained from the the Existence Theorem because our inequalities are
balanced.\newline
\indent For every $\epsilon>0$ we find an increasing sequence of integers $%
1=M_0 < L_1 < M_2 < L_2 < \dots $ and unitaries $(U_{M_{i+1}}) \in
A_{M_{i+1}}^{+}$, $(V_{i}^{n})_n \in B_{L_i}^{+}$ such that for $f \in
F_{M_{i}}$ and $g \in G_{L_{i}}$ we have 
\begin{equation*}
\frac{||U_{M_{i+1}}\tau^{\prime}_{M_i}(\tau_{M_i}(f))U^*_{M_{i+1}}-%
\phi_{M_iM_{i+1}}(f)||}{||f||}<\frac{\epsilon}{2^i},
\end{equation*}
\begin{equation*}
\frac{||V_{M_{i+1}}\tau_{L_i}(\tau^{\prime}_{L_i}(g))V^*_{L_{i+1}}-%
\phi_{L_iL_{i+1}}(g)||}{||g||}<\frac{\epsilon}{2^i}.
\end{equation*}
\indent In other words passing to suitable subsequences of algebras, it is
possible to perturb each of the homomorphisms obtained in the Existence
Theorem by an approximately inner automorphism, in such a way that the
diagram becomes an approximate intertwining, in the sense of Theorem 2.1, %
\cite{ell3}.\newline
\indent Therefore, by the Elliott approximate intertwining theorem (see \cite%
{ell3}, Theorem 2.1), the algebras $A$ and $B$ are isomorphic. }

{\large %%%%%%%%%%%%%%%%%%%%%%%%%%% 
}

{\large 
}

{\large %%%%%%%%%%%%%%%%%%%%%%
% the following local property: for any finite set $F \subset H$, for any $\eps%ilon >0$ there exists an $i$ and a hereditary sub-C*-algebra $H_i$ of $A_i$ su%ch that the distance between $F$ and $H_i$ is at most $\epsilon$.
%%%%%%%%%%%%%%%%%%%%%%%%
}

{\large %\end{prop}
%\begin{proof} 
}

{\large %%%%%%%%%%%%%%%%%%%%%%%%%%%%
%Without loss of generality let us assume that $F$ consists of one positive ele%ment $\{x\}$.
%%%%%%%%%%%%%%%%
}

{\large 
}

\section{\protect\large \ The range of the invariant}

{\large \indent In this section we prove Theorem 3.2 which answers the
question what are the possible values of the invariant from the isomorphism
theorem 3.1. It is useful to notice that the invariant consists of two
parts. One part is the stable part, i.e., $K_0$, AffT$^+$, $\lambda: T^+
\mapsto S(K_0)$ which was shown by K. Thomsen in \cite{kto} to be necessary
if one wants to construct an AI-algebra, and the other part which one may
call the non-stable part, namely Aff$^{\prime}$ or equivalently, as shown in %
\cite{ste0}, Remark 30.1.1 and Remark 30.1.2, the trace norm map. It is the
non-stable part of the invariant that one needs to investigate in its full
generality. Next the definition of the trace norm map is introduced. }

\begin{definition}
{\large Let $\mathcal{A}$ be a sub-C*-algebra of a $C^*$-algebra $\mathcal{B}$. The trace norm
map associated to $\mathcal{A}$ is a function $f:T^+(\mathcal{A}) \rightarrow (0,\infty]$ such
that $f( \tau) = || \tau|_{\mathcal{A}} ||$, $\infty$ if $\tau$ is unbounded. }
\end{definition}

{\large Recall that: }

\begin{definition}
{\large $T^+(\mathcal{A})$ is the cone of positive trace functionals on $\mathcal{A}$ with the
inherited w*-topology. }
\end{definition}

{\large %%%%%%%%%%%%%%%%%%%%%%%%%%%%%%%%%%%
%\begin{remark} Let ${\mathcal A}$ be a continuous-trace C*-algebra whose spect%rum is $[0,1]$. The trace norm map restricted to the extreme traces is exactly% the dimension function.
%\end{remark}
%%%%%%%%%%%%%%%%%%%%%%%%%%%%%
%\begin{remark} Knowing the trace norm map is equivalent to knowing the special% subset of the affine space, the scale, Aff$'$(). For instance if the trace no%rm map takes only the infinity value then the scale is the whole affine space %Aff$T^{+}()$.
%\end{remark}
%%%%%%%%%%%%%%%%%%%%%%
}

\begin{remark}
{\large The trace norm map is a lower semicontinuous affine map (being a
supremum of a sequence of continuous functions). }
\end{remark}

\begin{remark}
{\large The dimension range can be determined using the values of the trace
norm map $f$ , the simplex of tracial states $S$ and dimension group $G$. A
formula for the dimension range $D$ is: 
\begin{equation*}
D = \{ x\in G / v(x) < f(v) , v \in S , v \neq 0 \}
\end{equation*}
}
\end{remark}

{\large %%%%%%%%%%%%%%%%%%%%%%%%%%%%%
%\begin{remark} If the algebra ${\mathcal H}$ is unital then the trace norm map% does not take the infinity value.
%\end{remark}
%%%%%%%%%%%%%%%%%%%%%%%%%%%%%%%%%%%
I. Stevens has constructed a hereditary sub-C*-algebra of a simple (unital)
AI-algebra which is obtained as an inductive limit of hereditary
sub-C*-algebras of interval algebras, and has as a trace norm map any given
affine continuous function; cf.\cite{ste0}, Proposition 30.1.7. Moreover she
showed that any lower semicontinuous map can be realized as a trace norm map
in a special case. Our result is a generalization to the case of unbounded
trace norm map when restricted to the base of the cone. It is worth
mentioning that our approach gives another proof in the case of any lower
semicontinuous map as a trace norm map. Still our approach is using the I.
Stevens's proof for the case of continuous trace norm map. \newline
}

{\large \noindent \textbf{Theorem 3.2} \textit{\ Suppose that G is a simple
countable dimension group, V is the cone associated to a metrizable Choquet
simplex. Let $\lambda : V \rightarrow Hom^+(G,R)$ be a continuous affine map
and taking extreme rays into extreme rays. Let $f:V\rightarrow [0,\infty]$
be an affine lower semicontinuous map, zero at zero and only at zero. Then $%
[G, V, \lambda,f]$ is the Elliott invariant of some simple non-unital
inductive limit of continuous trace C*-algebras whose spectrum is the closed
interval $[0,1]$ or a finite disjoint union of closed intervals.} }

\begin{proof}
{\large The proof is based on I. Stevens's proof in a special case and
consists of several steps.\newline
\noindent \textbf{Step 0}\newline
\indent We start by constructing a simple stable AI algebra $\mathcal{A}$ with its
Elliott invariant: $[(G,D),V,\lambda ]$. We know that this is possible (see %
\cite{ste}). By tensoring with the algebra of compact operators we may
assume $\mathcal{A}$ is a simple stable AI algebra.\newline
\noindent \textbf{Step 1}\newline
\indent We restrict the map $f$ to some base $S$ of the cone $T^+(\mathcal{A})$,
where the cone $V$ is naturally identified with $T^+(\mathcal{A})$. Since any lower
semicontinuous affine map $f: S \rightarrow (0, +\infty]$ is a pointwise
limit of an increasing sequence of continuous affine positive maps, (see %
\cite{alf}), we can choose $f = \mathrm{lim} f_n$, where $f_n$ are
continuous affine and strictly positive functions. \newline
\indent Moreover by considering the sequence of functions $g_n = f_{n+1}-f_n$
if $n>1$ and $g_1 = f_1$ we get that: 
\begin{equation*}
\sum_{n=1}^{\infty}g_n = f
\end{equation*}
\textbf{Step 2}\newline
\indent Next we use the results of Stevens (\cite{ste0}, Prop. 30.1.7), to
realize each such continuous affine map $g_n$ as the norm map of a
hereditary sub-C*-algebra $\mathcal{B}_n$ (which is an inductive limit of special
algebra) of the AI algebra $\mathcal{A}$ obtained at Step 0.\newline
\indent Consider the $L^{\infty}$ direct sum $\oplus \mathcal{B}_i$ as a
sub-C*-algebra of $\mathcal{A}$. The trace norm map of the sub-C*-algebra 
$\oplus \mathcal{B}_i$ of $\mathcal{A}$ is equal to $\sum_{i=1}^{\infty}g_n = f$.\newline
\indent To see that $\oplus \mathcal{B}_i$ is a sub-C*-algebra of $\mathcal{A}$ we use that 
$\mathcal{A}$ is a stable C*-algebra:\newline
\begin{equation*}
\oplus \mathcal{B}_i = \left( 
\begin{array}{cccc}
\mathcal{B}_1 &  & 0 &  \\ 
& \mathcal{B}_2 &  &  \\ 
0 &  & \ddots & 
\end{array}
\newline
\right) \subseteq \mathcal{A}\otimes \mathbb{K} \cong \mathcal{A}. 
\end{equation*}
\indent Next denote with $\mathcal{H}$ the hereditary sub-C*-algebra generated by $%
\oplus \mathcal{B}_i$ inside of $\mathcal{A}$.\newline
\indent To prove that the trace norm map of $\mathcal{H}$ is $f$ is enough to show
that the norm of a trace on $\oplus \mathcal{B}_i$ is the same as on $\mathcal{H}$.\newline
\indent It suffices to prove that an approximate unit of the sub-C*-algebra 
$\oplus \mathcal{B}_i$ is still an approximate unit for the hereditary sub-C*-algebra 
$\mathcal{H}$.\newline
\indent We shall prove first that the hereditary sub-C*-algebra generated by 
$\oplus \mathcal{B}_i$ coincides with the hereditary sub-C*-algebra generated by one
of its approximate units. Let $(u_{\lambda})_{\lambda}$ be an approximate
unit of $\oplus \mathcal{B}_i$. Denote by $\mathcal{U}$ the hereditary sub-C*-algebra of 
$\mathcal{H}$ generated by $\{(u_{\lambda})_{\lambda}\}$. We want to prove 
that $\mathcal{U}$ is equal with $\mathcal{H}$. \newline
\indent Since $(u_{\lambda})_{\lambda}$ is a subset of $\oplus \mathcal{B}_i$ we
clearly have 
\begin{equation*}
\mathcal{U} \subseteq \mathcal{H}.
\end{equation*}
\indent For the other inclusion, one can observe that 
\begin{equation*}
\mathrm{for}\;\mathrm{all}\; b \in \oplus \mathcal{B}_i:\;b=\lim\limits_{\lambda
\rightarrow \infty} u_{\lambda}bu_{\lambda}. 
\end{equation*}
\indent Now each $u_{\lambda}bu_{\lambda}$ is an element of the hereditary
sub-C*-algebra generated by $(u_{\lambda})_{\lambda}$ and hence $b \in \mathcal{U}$.
Therefore $\oplus \mathcal{B}_i \subset \mathcal{U}$ which implies 
$\mathcal{H} \subseteq \mathcal{U}.$\newline
\indent We conclude that $\mathcal{H} = \mathcal{U}$ and hence the trace 
norm map of $\mathcal{H}$ is $f$. Therefore $\mathcal{H}$ is a simple 
hereditary sub-C*-algebra of an AI algebra with the prescribed invariant. }
\end{proof}

\begin{remark}
{\large \indent The approximate unit $(u_{\lambda})_{\lambda}$ of $\oplus \mathcal{B}%
_i$ is still an approximate unit for the hereditary sub-C*-algebra $\mathcal{U}$. To
see why this is true let us consider the sub-C*-algebra of $\mathcal{A}$ defined as
follows: $\{ h \in \mathcal{A} \;|\; h= \lim\limits_{\lambda \rightarrow
\infty}u_{\lambda}h \}$.\newline
\indent This sub-C*-algebra of $\mathcal{A}$ is a hereditary sub-C*-algebra. Indeed
let $0 \leq k \leq h $ with $h= \lim\limits_{\lambda \rightarrow
\infty}u_{\lambda}h$. We want to prove that $k = \lim\limits_{\lambda
\rightarrow \infty}u_{\lambda}k$.\newline
\indent Consider the hereditary sub-C*-algebra $\overline{h\mathcal{A}h}$ of $\mathcal{A}$
which clearly contains $h$ (because $h^2 = \lim\limits_{\lambda \rightarrow
\infty }hu_{\lambda}h$). Therefore $k \in \overline{h\mathcal{A}h}$.\newline
\indent Since $h= \lim\limits_{\lambda \rightarrow \infty}u_{\lambda}h$ we
obtain that $u_{\lambda}$ is an approximate unit for $\overline{h\mathcal{A}h}$. In
particular 
\begin{equation*}
k= \lim\limits_{\lambda \rightarrow \infty}u_{\lambda}k
\end{equation*}
and hence $\{ h \in \mathcal{A} \;| h= \lim\limits_{\lambda \rightarrow
\infty}u_{\lambda}h \}$ is a hereditary sub-C*-algebra of $\mathcal{A}$. Since $\mathcal{U}$
is the smallest hereditary containing $(u_{\lambda})_{\lambda}$ we get that 
\begin{equation*}
\mathcal{U} \subseteq \{ h \in \mathcal{A} \;| h= \lim\limits_{\lambda \rightarrow
\infty}u_{\lambda}h \} 
\end{equation*}
and $u_{\lambda}$ is an approximate unit for $\mathcal{U}$. }
\end{remark}

\section{\protect\large Non-AI algebras which are inductive limits of
continuous-trace C*-algebras}

{\large In this section we present a necessary and sufficient condition 
on the invariant for the algebra to be AI. We shall use this in the next section 
to construct an inductive limit of continuous trace
C*-algebras with spectrum $[0,1]$ which is not an AI algebra.\newline
\indent With $[G,V, \lambda,f]$ as before we observe that for an AI algebra
with Elliott invariant canonically isomorphic to the given invariant 
the following equality always holds: 
\begin{equation*}
f(v) = \mathrm{sup} \{v(g): g \in D \},
\end{equation*}
where D is the dimension range. This is seen by simply using the fact that
any AI algebra has an approximate unit consisting of projections.\newline
\indent Therefore a sufficient condition imposed on the invariant in order
to get an inductive limit of continuous trace C*-algebra with spectrum $[0,1]
$ but not an AI algebra is  
\begin{equation*}
f(v) \neq \mathrm{sup} \{v(g): g \in D \}. 
\end{equation*}
\indent This condition is also necessary. Namely assume that we have $f(v) = 
\mathrm{sup} \{v(g): g \in D \}$ and we have constructed a simple C*-algebra 
$\mathcal{A}$ which is an inductive limit of continuous trace C*-algebras with
spectrum $[0,1]$ and with the invariant canonically isomorphic with the
tuple $[G,V, \lambda,f]$. Consider $D = \{x \in G : v(x) < f(v), v \in S, v
\neq 0\}$, where $S$ is a base of the cone $V$. For the tuple $%
[G,D,V,S,\lambda]$ we can build (via the range of the invariant for simple
AI algebras, \cite{ste}) a simple AI-algebra $\mathcal{B}$ with the invariant
naturally isomorphic with the given tuple. \newline
\indent Note that the trace norm map which is defined starting from the
tuple \newline
$[K_{0}(\mathcal{B}),D(\mathcal{B}),T^{+}\mathcal{B}, \lambda_{\mathcal{B}}]$ 
is exactly $f$ because of the equality
\begin{equation*}
f(v) = \mathrm{sup} \{v(g): g \in D \}
\end{equation*}
and $\mathcal{B}$ is an AI algebra.
\newline
\indent It is clear that $\mathcal{B}$ is an inductive limit of continuous trace
C*-algebras with spectrum $[0,1]$ and hence by the isomorphism theorem 2.1
we conclude that $\mathcal{A}$ isomorphic to $\mathcal{B}$. Hence $\mathcal{A}$ 
is a simple AI algebra
as desired and we have proved the following theorem: }

\begin{theorem}
{\large Let $\mathcal{A}$ be a simple C*-algebra which is an inductive limit of
continuous-trace C*-algebras whose spectrum is homeomorphic to $[0,1]$. A
necessary and sufficient condition for $\mathcal{A}$ to be a simple AI algebra is
\begin{equation*}
f(v) = \mathrm{sup} \{v(g): g \in D \}. 
\end{equation*}
}
\end{theorem}

\section{\protect\large \ The class of simple inductive limits of continuous
trace C*-algebras with spectrum $[0,1]$ is much larger than the class of
simple AI algebras}

{\large To see this consider the simple AI algebra necessarily not of real
rank zero with scaled dimension group $(\mathbb{Q},\mathbb{Q}_{+})$ and cone
of positive trace functionals a 2-dimensional cone; see \cite{ste}. Then the
set of possible stably AI algebras, or equivalently the set of possible
trace norm maps, may be represented as the extended affine space shown in
the following schematic diagram: }

\begin{center}
\begin{picture}(400,200)(-50,0)
\put(50,10){\dashbox{5}(180,180)}
\put(50,10){\line(1,1){180}}
\linethickness{3pt}
\put(230,10){\line(0,1){180}}
\put(50,190){\line(1,0){180}}
\end{picture}

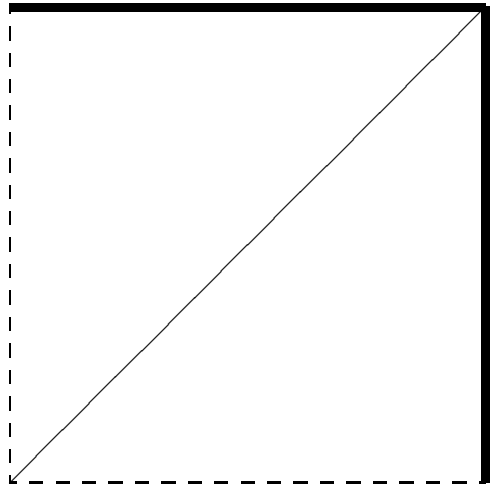
Figure 6.  
\end{center}

%\begin{center}
%{\large \includegraphics [height =3cm]{range.eps}\\[0pt]
%Figure 6. }
%\end{center}

{\large \indent Each off-diagonal point in the diagram is the trace norm map
of one of I. Stevens's algebras. The boundary points of the first quadrant
are removed (dotted lines) and the points with infinite coordinates are
allowed. The dimension range is embedded in a canonical way in the extended
affine space as the main diagonal consisting of the points with rational
coordinates.\newline
\indent The two bold lines represent the cases of inductive limits of
continuous trace C*-algebras with unbounded trace norm map (points on these
two lines have at least one coordinate infinity).\newline
\indent If the point is off the diagonal and in the first quadrant, by
Theorem 12.1 we get that the corresponding C*-algebra is an inductive limit
of continuous trace C*-algebras which is not AI-algebra. It is clear that
the size of the set of points off the diagonal is much larger then the size
of the set of points on the diagonal. (For instance in terms of the Lebesgue
measure.)\newline
\indent This picture shows that the class of simple AI algebras sits inside
the class of inductive limits of continuous trace C*-algebras in the same
way that the main diagonal sits inside the first quadrant. 
%%%%%%%%%%%%%%%%%%%%%%%%%%
%(Note especially the small dots which correspond to both unital and non-unital% AI algebras: recall the case of stably UHF algebras.)\\
%\indent The algebras in the preceding example are not of real rank zero.
%%%%%%%%%%%%%%%%%%%%%%%%
}

\end{document}